\documentclass[11pt, reqno]{amsart}
\usepackage{amssymb, amsthm, amsmath, amsfonts}
\usepackage{graphics}
\usepackage{hyperref}
\usepackage[all]{xy}
\usepackage{enumerate}
\usepackage[mathscr]{eucal}
\usepackage{enumerate}
\usepackage{thmtools}
\usepackage{xcolor}
\usepackage{ytableau}

\theoremstyle{plain}
\newtheorem{theorem}{Theorem}[section]
\newtheorem{lemma}[theorem]{Lemma}
\newtheorem{proposition}[theorem]{Proposition}

\newtheorem{corollary}[theorem]{Corollary}
\numberwithin{equation}{section}

\theoremstyle{definition}

\newtheorem{definition}[theorem]{Definition}
\newtheorem{example}[theorem]{Example}
\newtheorem{remark}[theorem]{Remark}

\DeclareMathOperator{\HH}{H}
\DeclareMathOperator{\Mor}{Mor}
\DeclareMathOperator{\Mod}{-Mod}
\DeclareMathOperator{\DMod}{-Mod^{\mathrm{dis}}}
\DeclareMathOperator{\module}{-mod}
\DeclareMathOperator{\dmodule}{-mod^{\mathrm{dis}}}
\DeclareMathOperator{\fdmod}{-mod^{\mathrm{fd}}}
\DeclareMathOperator{\Ob}{Ob}

\DeclareMathOperator{\Ext}{Ext}
\DeclareMathOperator{\Hom}{Hom}
\DeclareMathOperator{\Stab}{Stab}

\newcommand{\fHom}{\mathscr{H}om}
\newcommand{\fExt}{\mathscr{E}xt}
\newcommand{\BC}{\boldsymbol{\C}}
\newcommand{\C}{{\mathscr{C}}}
\newcommand{\D}{{\mathscr{D}}}

\newcommand{\N}{{\mathbb{N}}}
\newcommand{\RR}{{\underline{R}}}

\newcommand{\FI}{{\mathrm{FI}}}
\newcommand{\OI}{{\mathrm{OI}}}
\newcommand{\VI}{{\mathrm{VI}}}
\newcommand{\FS}{{\mathrm{FS}}}
\newcommand{\Sets}{{\mathrm{Set}}}
\newcommand{\Aut}{{\mathrm{Aut}}}
\newcommand{\op}{{\mathrm{op}}}
\newcommand{\tor}{{\mathrm{tor}}}
\newcommand{\sat}{{\mathrm{sat}}}
\newcommand{\Sh}{{\mathrm{Sh}}}
\newcommand{\PSh}{{\mathrm{PSh}}}

\title[Sheaves of modules and discrete representations of topological groups]{Sheaves of modules on atomic sites and discrete representations of topological groups}

\author{Zhenxing Di}
\address{School of Mathematical Sciences, Huaqiao University, Quanzhou, Fujian, 362021, China.}
\email{dizhenxing@163.com}

\author{Liping Li}
\address{School of Mathematics and Statistics, Hunan Normal University, Changsha, Hunan 410081, China.}
\email{lipingli@hunnu.edu.cn}

\author{Li Liang}
\address{Department of Mathematics, Lanzhou Jiaotong University, Lanzhou, Gansu 730070, China.}
\email{lliangnju@gmail.com}

\author{Fei Xu}
\address{Department of Mathematics, Shantou University, Shantou, Guangdong 515063, China.}
\email{fxu@stu.edu.cn}

\thanks{Z. Di is supported by the National Natural Science Foundation of China (grant No. 11971388); L. Li is supported by the National Natural Science Foundation of China (grant No. 11771135) and the Hunan Provincial Science and Technology Department (grant No. 2019RS1039); L. Liang is supported by the National Natural Science Foundation of China (grant No. 11761045); F. Xu is supported by the National Natural Science Foundation of China (grant No. 11672145).}

\keywords{Atomic Grothendieck topology, sheaves, sheafification, sheaf cohomology, Serre quotient, Nakayama functor, topological groups, discrete representations.}

\begin{document}

\begin{abstract}
The main goal of this paper is to establish close relations among sheaves of modules on atomic sites, representations of categories, and discrete representations of topological groups. We characterize sheaves of modules on atomic sites as saturated representations, which are precisely representations right perpendicular to torsion representations in the sense of Geigle and Lenzing. Consequently, the category of sheaves is equivalent to the Serre quotient of the category of presheaves by the category of torsion presheaves. We also interpret the sheaf cohomology functors as derived functors of the torsion functor and for some special cases as the local cohomology functors. These results as well as a classical theorem of Artin provides us a new approach to study discrete representations of topological groups. In particular, by importing established facts in representation stability theory, we explicitly classify simple or indecomposable injective discrete representations of some topological groups such as the infinite symmetric group, the infinite general or special linear group over a finite field, and the automorphism group of the linearly ordered set $\mathbb{Q}$. We also show that discrete representations $V$ of these topological groups satisfy a certain stability property described in Theorem \ref{stability of discrete representations}.
\end{abstract}

\maketitle

\section{Introduction}

\subsection{Motivation}
Topos theory studies sheaves of sets over (small) categories equipped with Grothendieck topologies, significantly formalizing and generalizing sheaves on topological spaces. It connects many theories in very different languages and hence provides a common foundation for various mathematical areas such as first order logic, model theory, algebraic geometry, topology, functional analysis, and algebra \cite{Jo}. For the interest of the authors, topos theory also provides a powerful tool for studying representations of groups and categories. An important contribution along this approach is established by Artin, who proved a fundamental result relating discrete $G$-sets of topological groups $G$ and sheaves of sets over particular orbit categories equipped with \textit{atomic topology} (defined in Subsection \ref{atomic topology}); see \cite{Ar}, \cite[II.1.9]{Mi} or \cite[III.9, Theorems 1 and 2]{MM}. In particular, if we take $G$ to be infinite symmetric group over countably many elements, then the associated orbit category is precisely $\FI^{\op}$, the opposite category of the category of finite sets and injections, and the category of sheaves of sets over it with atomic topology is the famous \textit{Schanuel topos} in categorical logic. Thus the result of Artin connects quite a few areas, including topos theory, representation stability theory introduced by Church, Ellenberg and Farb in \cite{CEF}, and continuous actions of topological groups on sets. This strategy has also been used to study representations of finite groups in \cite{WX, XX}. Explicitly, given a finite group $G$, one may construct a few small categories $\mathscr{C}$ and equip them with various Grothendieck topologies. It is found that sheaves of modules over these categories are closely related to representations of $G$, and categories of sheaves correspond to various abelian subcategories of the representation category of $G$. Consequently, one may investigate representation theory of $G$ via considering sheaves of modules over $\mathscr{C}$.

These results motivate us to explore more connections among continuous representations of topological groups, sheaves of modules over their orbit categories, and representations of these categories. This is a two-fold project. On one side, since Artin's theorem asserts that the category of \textit{discrete representations} (Defined in Section \ref{discrete repns section}) of a topological group $G$ is equivalent to the category of sheaves of modules over an orbit category $\C$ equipped with the atomic topology, we hope to obtain similar results for other continuous representations of $G$ and other Grothendieck topologies. On the other side, one hopes to characterize these sheaf categories as various subcategories or quotient categories of the representation category of $\C$.

Let us give a few more details about the second side. Let $\C$ be a skeletally small category and $R$ a commutative ring. A \textit{presheaf} of $R$-modules is a contravariant functor from $\C$ to the category $R \Mod$ of $R$-modules, or equivalently, a covariant functor from $\C^{\op}$ to the category $R \Mod$ (we sometimes call it a \textit{$\C^{\op}$-module}, or a \textit{representation of $\C^{\op}$}). When we impose a Grothendieck topology on $\C$, a presheaf of $R$-modules is a \textit{sheaf} with respect to this topology if it satisfies several equivalent axioms; see for instances \cite[III.4, Proposition 1]{MM}. In other words, sheaves of $R$-modules over $\C$ are representations of $\C^{\op}$ satisfying certain special requirements (which are, of course, dependent on the imposed topology), and the category of sheaves is a full subcategory of the category of $\C^{\op}$-modules. However, these special requirements, though are explicit in a formal sense, are hard to check in practice. Therefore, for experts in other areas (especially representation theory) who hope to apply results in sheaf theory to their research, it is desirable to find an equivalent but more accessible approach; that is, to interpret concepts and results in sheaf theory in the language of representation theory of categories. In particular, there are several important questions to be answered: finding representation theoretic conditions such that a $\C^{\op}$-modules is a sheaf of $R$-modules over $\C$ if and only if it satisfies these conditions; reformulating the \textit{sheafification functor} in terms of natural functors in representation theory; finding equivalences between the category of sheaves and certain special module categories of $\C^{\op}$.

In this paper we begin the above mentioned project for a special Grothendieck topology: the atomic topology. We choose it as our first candidate because of the following two reasons. Firstly, Artin's theorem has provided a completely satisfactory solution for the first side, so we only need to establish relations between sheaves of modules over atomic sites and representations of orbit categories. Secondly, representations of orbit categories of many topological groups, such as the infinite symmetric group and the infinite general or special linear group over a finite field, have been well understood by a series of works in \cite{CEF, CEFN, GLX, LR, LY, Nag1, Nag2, PW, PS, SS, SS2, SS3, Wil}, which can be applied to provide significant results and insights for discrete representations of these topological groups.

The main contributions of this paper can be summarized as follows. We characterize sheaves of modules on atomic sites as \textit{saturated modules} defined by the vanishing of the torsion functor and its first right derived functor (see Definition \ref{saturated modules}). It follows immediately from \cite[Proposition 2.2]{GLen} that the category of sheaves is equivalent to the Serre quotient of the category of presheaves by the category of torsion presheaves. We also interpret the sheafification functor and sheaf cohomology functors in terms of localization functors, section functors, and derived functors of the torsion functor in representation theory. Combining these results and Artin's theorem, we deduce that the category of discrete representations of an arbitrary topological group $G$ is equivalent to the Serre quotient of the category of presheaves over a particular orbit category of $G$ by the category of torsion presheaves. As applications, we consider certain topological groups $G$ (including the infinite symmetric group, the infinite general or special linear group over a finite field, and the automorphism group of the linearly ordered set $\mathbb{Q}$), and explicitly construct their associated orbit categories (which turn out to be well known combinatoric categories $\FI$, $\VI$ and $\OI$ widely studied in representation stability theory). By applying established results in representation stability theory, we describe the structure of categories of discrete representations of $G$, and classify simple and indecomposable injective discrete representations. Although a few results on discrete representations of these topological groups have been described in the literature (in particular, \cite{Nag2, SS2}), our approach via sheaf theory seems more uniform and provides extra theoretic insights. Finally, we describe an interesting stability property for discrete representations of the above topological groups.

Now we begin to describe the work of this paper in details.

\subsection{From sheaf theory to representation theory}

Suppose that $\C$ is a skeletally small category and satisfies the right Ore condition. Then one may impose the atomic topology $J_{at}$ on $\C$ to get a site $\boldsymbol{\C} = (\C, J_{at})$. In this situation, the constant structure presheaf $\underline{R}$ assigning to each object in $\C$ the ring $R$ and to each morphism in $\C$ the identity map $\mathrm{id}_R$ is a sheaf of rings, and sheaves of $\RR$-modules over $\boldsymbol{\C}$  \cite{KS, Stack} are precisely presheaves of $R$-modules (defined as above) such that the underlying presheaves of sets are sheaves of sets.

On the other hand, one can consider a natural torsion theory in $\C^{\op} \Mod$. A $\C^{\op}$-module $V$ is \textit{torsion} if for every object $x$ and every element $v \in V(x)$, there is a morphism $\alpha: x \to y$ in $\C^{\op}$ such that $V(\alpha): V(x) \to V(y)$ sends $v$ to 0. For every $\C^{\op}$-module $V$ there is a natural short exact sequence $0 \to V_T \to V \to V_F \to 0$ such that $V_T$ is the maximal torsion submodule of $V$. Note that the assignment $V \mapsto V_T$ gives rise to a left exact functor $\tau$. We say that $V$ is \textit{torsion free} if $V_T = \tau V = 0$, and $V$ is \textit{saturated} if $\mathrm{R}^i \tau (V) = 0$ for $i=0, 1$, where $\mathrm{R}^i \tau$ is the $i$-th right derived functor of $\tau$. In the language of \cite{GLen}, the full subcategory of saturated modules is the right perpendicular category of the category of torsion modules.

Since the category $\C^{\op} \Mod^{\tor}$ of torsion $\C^{\op}$-modules is a Serre subcategory of $\C^{\op} \Mod$, one may define the Serre quotient $\C^{\op} \Mod / \C^{\op} \Mod^{\tor}$. It turns out that for atomic sites, one may translate notions on sheaves of $R$-modules in terms of the above notions in representation theory. As the first main result, for the atomic topology, we have the following ``translation" from sheaf theory to representation theory:

\begin{center}
\begin{tabular}{c|c}
 Sheaf theory & Representation theory\\
  \hline
presheaves (of $R$-modules) over $\C$ & $\C^{\op}$-modules\\
separated presheaves over $\C$ & torsion free $\C^{\op}$-modules\\
sheaves (of $R$-modules) over $\BC=(\C,J_{at})$ & saturated $\C^{\op}$-modules\\
presheaf category $\PSh(\C,R)$ & $\C^{\op} \Mod$\\
sheaf category $\RR\Mod=\Sh(\BC, R)$ & Serre quotient $\C^{\op} \Mod / \C^{\op} \Mod^{\tor}$\\
sheaf comohomology functors & right derived functors of the torsion functor\\
\end{tabular}
\end{center}

More precisely, we have:

\begin{theorem} \label{first main theorem}
Let $\C$ be a skeletally small category which satisfies the right Ore condition and is equipped with the atomic topology (see Subsection \ref{atomic topology}), and let $R$ be a commutative ring. Then:
\begin{enumerate}
\item A presheaf $V$ of $R$-modules over $\C$ is a sheaf if and only if $V$ is saturated as a $\C^{\op}$-module.

\item The category $\Sh(\BC, R)$ of sheaves of $R$-modules over $\BC$ is equivalent to the Serre quotient category $\C^{\op} \Mod / \C^{\op} \Mod^{\tor}$.

\item For a sheaf $V \in \Sh(\BC, R)$, an arbitrary object $x$ in $\C$, and $i \geqslant 1$, the sheaf cohomology group $\mathrm{R}^i \Gamma_x(V)$ on $x$ is isomorphic to $(\mathrm{R}^{i+1} \tau (V))_x$, the value of the $\C^{\op}$-module $\mathrm{R}^{i+1} \tau (V)$ on $x$.
\end{enumerate}
\end{theorem}

\begin{remark} \normalfont
As we pointed out before, statement (2) is a direct consequence of \cite[Proposition 2.2]{GLen}, but statement (1) is not obvious from that proposition. Moreover, although in the above theorem we only deal with sheaves of modules over the constant structure presheaf $\underline{R}$ of rings sending each object in $\C$ to $R$, the characterization specified in the first statement of this theorem also applies to presheaves of modules $V$ over arbitrary structure presheaves $\mathcal{O}$. Indeed, $V$ is a sheaf of $\mathcal{O}$-modules if and only if it is a presheaf of $\mathcal{O}$-modules and also a sheaf of abelian groups (see Subsection \ref{sheaves of modules}). However, a sheaf of abelian groups is nothing but a sheaf of modules over the constant structure presheaf $\underline{\mathbb{Z}}$, and hence degenerates to the case described in the first statement of this theorem.
\end{remark}

\subsection{A more transparent description via Nakayama functor}

Objects and the structure of the Serre quotient category $\C^{\op} \Mod /\C^{\op} \Mod^{\tor}$ are still mysterious for the purpose of application. One thus may ask for more explicit descriptions under certain special circumstances. We consider two types of categories equipped with suitable combinatorial structure. In particular, in the situation that $R$ is a field of characteristic 0, and that $\C^{\op}$ has the \textit{locally Noetherian} property (that is, the category $\C^{\op} \module$ of finitely generated $\C^{\op}$-modules over $R$ is abelian), one automatically obtains that the category $\mathrm{sh}(\BC, R)$ of ``finitely generated" objects in $\Sh(\BC, R)$ is equivalent to the Serre quotient $\C^{\op} \module / \C^{\op} \module^{\tor}$ formed by restricting to the categories of finitely generated objects. This quotient category already contains useful information which is of great interest to us; for example, all simple sheaves are contained in it. We then apply the strategy and the \textit{Nakayama functor} introduced in \cite{GLX} (see also Subsection \ref{type ii cats}) to establish the following result for quite a few infinite combinatorial categories $\C$ (including the category $\FI$, the category $\VI$ of finite dimensional vector spaces over a finite field and linear injections, the category $\OI$ of finite linearly ordered sets and order-preserving injections, etc):

\begin{theorem} \label{second main theorem}
Let $R$ be a field of characteristic 0, and let $\C^{\op}$ be a type II combinatorial category (see Subsection \ref{type ii cats}). Then the Nakayama functor (see Subsection \ref{type ii cats}) induces the following equivalence:
\[
\mathrm{sh} (\BC, R) \simeq \C^{\op} \module / \C^{\op} \module^{\tor} \simeq \C^{\op} \fdmod
\]
where $\C^{\op} \fdmod$ is the category of finite dimensional $\C^{\op}$-modules.
\end{theorem}

We remind the reader that the dimension of a $\C^{\op}$-module $V$ is the (possibly infinite) sum of the dimensions $\dim_R V_x$, where $V_x$ is the value of $V$ on $x$ and $x$ ranges over all objects in $\C$. Therefore, $V$ is of finite dimensional if and only if each $V_x$ is of finite dimensional and all but finitely many $V_x$ are 0. Consequently, we know that objects in $\mathrm{sh} (\BC, R)$ have finite length and finite injective dimension. Furthermore, since it is much easier to classify simple objects in $\C^{\op} \fdmod$, via this equivalence, one can explicitly construct all simple sheaves in $\mathrm{sh} (\BC, R)$.

\subsection{Discrete representations of topological groups}

Let $G$ be a topological group and let $V$ be an $RG$-module. We say that $V$ is a \textit{discrete representation} of $G$ if $V$ is equipped with the discrete topology and the action of $G$ on $V$ is continuous. By the canonical result of Artin (\cite[III.9, Theorems 1 and 2]{MM}), one can construct an orbit category $\mathbf{S}_{\mathcal{U}}G$ with respect to a cofinal system $\mathcal{U}$ of open subgroups of $G$, and the category $RG \DMod$ of discrete representations of $G$ is equivalent to the category $\Sh(\mathbf{S}_{\mathcal{U}}G, {R})$; for details, see Subsection \ref{Artin's theorem}. This result provides a novel approach to study representations of topological groups. In particular, combining Theorem \ref{first main theorem} and Artin's theorem (Theorem \ref{mm}), we obtain the following bridge connecting discrete representations of topological groups, sheaves of modules over atomic sites, and representations of categories:

\begin{corollary} \label{main corollary}
Let $G$ be an arbitrary topological group, let $R$ be a commutative ring, and let $\C = \mathbf{S}_{\mathcal{U}}G$  be the orbit category associated to a cofinal system $\mathcal{U}$ of open subgroups and equipped with the atomic topology. Then one has the following equivalences:
\[
RG \DMod \simeq \Sh(\BC, R) \simeq \C^{\op} \Mod / \C^{\op} \Mod^{\tor}.
\]
\end{corollary}

\begin{remark} \normalfont
Objects in the Serre quotient $\C^{\op} \Mod / \C^{\op} \Mod^{\tor}$ are called \textit{generic representations} by some people, and they also realized the equivalence between the first category and the third category in the above corollary for some special examples; see for instance \cite[Subsection 1.3]{Nag2}. It becomes clear by Artin's theorem and Theorem \ref{first main theorem} that this equivalence actually holds for \textbf{all} topological groups.
\end{remark}

Now we describe a natural construction, which unifies quite a few topological groups whose orbit categories with respect to a suitable choice of cofinal systems of open subgroups are well known combinatorial categories widely studied in representation stability theory. Let $X$ be an infinitely countable set and impose the discrete topology on it. Then the monoid $M(X)$ consisting of all maps from $X$ to itself can be equipped with the product topology (the usual pointwise convergence topology). Let $G$ be a subgroup of $M(X)$ consisting of invertible maps respecting some special rule (for instance, preserving the linear structure or the linearly ordered structure on $X$). Then $G$ inherits the topology on $M(X)$ as a subspace, and hence becomes a topological group. The following table includes a few interesting examples of $G$ where $X$ is $\mathbb{N}$, $\mathbb{F}_q^{\mathbb{N}}$, $\mathbb{Q}$, or $B_{\infty}$ (the free Boolean algebra on a countable infinity of generators) respectively:
\begin{center}
\begin{tabular}{c|ccc}
  $G$ & $\C = \mathbf{S}_{\mathcal{U}}G$ & objects & morphisms\\
  \hline
  $\mathrm{Aut}(\mathbb{N})$ & $\FI^{\op}$ & finite sets & injections \\
  $\varinjlim_n S_n$ & $\FI^{\op}$ & finite sets & injections \\
  $\mathrm{GL}(\mathbb{F}_q^{\mathbb{N}})$ & $\VI^{\op}$ & finite diml. spaces over $\mathbb{F}_q$ & linear injections \\
  $\varinjlim_n \mathrm{GL}_n(\mathbb{F}_q)$ & $\VI^{\op}$ & finite diml. spaces over $\mathbb{F}_q$ & linear injections \\
  $\varinjlim_n \mathrm{SL}_n(\mathbb{F}_q)$ & $\VI^{\op}$ & finite diml. spaces over $\mathbb{F}_q$ & linear injections \\
  $\Aut(\mathbb{Q}, \leqslant)$ & $\OI^{\op}$ & finite linearly ordered sets & order-preserving injections \\
  $\Aut(B_{\infty})$ & $\mathrm{FS}$ & finite sets & surjections.
\end{tabular}
\end{center}

Sheaf theory over categories in the above table has been widely applied in categorical logic theory; see for instance \cite[Examples D3.4.1, D3.4.11, and D3.4.12]{Jo}. More recently, their representations were also extensively investigated in representation stability theory. People have obtained a deep and comprehensive understanding of representation theoretic and homological properties via a series of works; see for instances \cite{CE, CEF, CEFN, DPV, DTV, GL2, GLX, GS, LR, LY, Nag, Nag1, Nag2, PS, SS, SS2}. The equivalence $RG \DMod \cong \C^{\op} \Mod / \C^{\op} \Mod^{\tor}$ by Corollary \ref{main corollary} thus allows us to study discrete representations of these topological groups via applying methods and results in representation stability theory. In particular, when $R$ is a field of characteristic 0 (or any field for $\OI^{\op}$), it has been proved in \cite{GL1, GL2, GS, SS} that the categories $\FI$, $\VI$ and $\OI$ satisfy the assumption in Theorem \ref{second main theorem}, so one can use it to classify simple sheaves and simple discrete representations of the corresponded topological groups.

In a summary, we obtain the following classification of simple discrete representations or simple sheaves.

\begin{theorem} \label{classfication}
Let $R$ be a field of characteristic 0 and $p$ a prime number. Then:
\begin{enumerate}
\item Simple discrete representations of $\mathrm{Aut}(\mathbb{N})$ or $\varinjlim_n S_n$ are parameterized by the set $\sqcup_{n \in \mathbb{N}} \mathcal{P}_n$, the set of partitions of $[n]$.
\item Simple discrete representations of $\mathrm{GL}(\mathbb{F}_q^{\mathbb{N}})$, $\varinjlim_n \mathrm{GL}_n(\mathbb{F}_q)$, or $\varinjlim_n \mathrm{SL}_n(\mathbb{F}_q)$ are parameterized by the set $\sqcup_{n \in \mathbb{N}} \mathrm{Irr} (\mathrm{GL}_n (\mathbb{F}^q))$, the set of isomorphism classes of irreducible representations of $\mathrm{GL}_n(\mathbb{F}_q)$.
\item Simple discrete representations of $\mathrm{Aut}(\mathbb{Q}, \leqslant)$ are parameterized by $\mathbb{N}$.
\item Let $\C = \mathcal{Z}(p^n)$ (resp., $\C = \mathcal{Z}(p^{\infty})$) be the category of finitely generated $\mathbb{Z}/p^n \mathbb{Z}$-modules (resp., finite abelian $p$-groups) and surjective module homomorphisms (resp., conjugacy classes of surjective group homomorphisms). Then simple objects in $\Sh(\BC, R)$ are parameterized by the set $\sqcup_{n \in \mathbb{N}} \mathrm{Irr} (\Aut(H))$, where $H$ ranges over all finite abelian groups of exponent dividing $p^n$ (resp., all finite abelian $p$-groups).
\end{enumerate}
\end{theorem}

\begin{remark} \normalfont
Sam and Snowden gave in \cite[Proposition 6.1.5]{SS2} a parametrization of irreducible \textit{algebraic representations} of $\varinjlim_n S_n$ over the complex field. In \cite[Theorem 1.10]{Nag2}, Nagpal proved that their parameterization actually holds for arbitrary fields. Moreover, he established in \cite[Theorem 1.12]{Nag2} a parametrization of irreducible \textit{admissible representations} of $\varinjlim_n \mathrm{GL}_n(\mathbb{F}_q)$ for any field in which $q$ is invertible. By carefully checking definitions, it is not hard to see that algebraic representations in \cite{SS2} or admissible representations in \cite{Nag2} are precisely discrete representations in this paper. Therefore, they have established the parameterizations for $\varinjlim_n S_n$ and $\varinjlim_n \mathrm{GL}_n(\mathbb{F}_q)$, while Nagpal's results are even more general. Other parameterizations given in this theorem, as far as the authors know, do not appear in the literature. Furthermore, for $\mathrm{Aut}(\mathbb{Q}, \leqslant)$, the above parametrization holds for any field; for $\mathcal{Z}(p^n)$ and $\mathcal{Z}(p^{\infty})$, the above parameterizations hold for any field of characteristic distinct from $p$.

Algebraic representations of the infinite general linear group were also considered in \cite{SS2}. Their work and the work described in this paper look quite similar, but are different. Actually, what they considered are algebraic complex representations of $\varinjlim_n \mathrm{GL}_n (\mathbb{C})$, while in this paper we consider discrete representations (over any field of characteristic 0) of $\varinjlim_n \mathrm{GL}_n (\mathbb{F}_q)$.

Representation theory of several examples in this theorem were studied by Harman and Snowden in \cite{HS}, where they propose a novel theory of integration on oligomorphic groups, and investigate the pre-Tannakian structure of rigid tensor categories of permutation modules in details.
\end{remark}

\begin{remark}
Isomorphism classes of irreducible representations of $\mathrm{GL}_n(\mathbb{F}_q)$ have been classified by Zelevinsky in \cite[Section 9]{Ze} when $R$ is the complex field. The categories $\mathcal{Z}(p^n)$ and $\mathcal{Z}(p^{\infty})$ were introduced in a recent paper \cite{Pol} by Pol and Strickland. Automorphism groups $\Aut(H)$ of finite abelian $p$-groups $H$ have been described in \cite{HR}. However, it is still unknown to the authors whether isomorphism classes of irreducible representations of $\Aut(H)$ have been classified for fields of characteristic 0.
\end{remark}

Stability patterns of representations, firstly observed by Church, Ellenberg, Farb and Nagpal for $\FI$ in \cite{CEF, CEFN}, are of significant interest in representation theory of categories. Because of the tight relation between sheaves of modules over atomic sites and discrete representations of topological groups, it is not surprising that this phenomena also occurs for the latter ones.

\begin{theorem} \label{stability of discrete representations}
Let $G$ be one of the following topological groups:
\[
\Aut(\N), \, \varinjlim_n S_n, \, \mathrm{GL}(\mathbb{F}_q^{\N}), \, \varinjlim_n \mathrm{GL}_n(\mathbb{F}_q), \, \varinjlim_n \mathrm{SL}_n(\mathbb{F}_q), \, \Aut(\mathbb{Q}, \leqslant), \, \Aut(B_{\infty}),
\]
and let $V$ be a discrete representation of $G$ over a commutative ring $R$. If the presentation degree $N_V$ of $V$ is finite (see \ref{stability result} for its definition), then
\[
V = \bigcup_{\substack{T \subset X \\ |T| \leqslant N_V }} V^{U_T}.
\]
where $T$ is a subset of $X = \N$ for the first two groups, a subset of $X = \mathbb{F}_q^{\N}$ for the three linear groups, a subset of $X = \mathbb{Q}$ for $G = \Aut(\mathbb{Q}, \leqslant)$, a subalgebra of $X = B_{\infty}$ for the last group,
\[
U_T = \Stab_G(T) = \{ g \in G \mid g \cdot x = x \text{ for all } x \in T \}
\]
and
\[
V^{U_T} = \{ v \in V \mid g \cdot v = v \text{ for all } g \in U_T \}.
\]
Furthermore, $N_V$ is the minimal number such that the above identity holds.
\end{theorem}

\begin{remark}
Since $V$ is a discrete representation of $G$, for every element $v \in V$, $\Stab_G(v) \leqslant G$ is an open subgroup. Note that $G$ has a cofinal system $\mathcal{U}$ consisting of open subgroups $\Stab_G(T)$ with $T$ a finite subset of $X$, so we can find a finite subset $T_v \subset X$ such that $\Stab_G(T_v) \leqslant \Stab_G(v)$. In particular, $U_{T_v} = \Stab_G(T_v)$ fixes $v$, so one has
\[
V = \bigcup_{T_v} V^{U_{T_v}}.
\]
But in general one cannot expect a smallest common finite upper bound (called the \textit{stable range} of $V$) for the cardinalities $|T_v|$. However, the theorem asserts that for a discrete representation presented in finite degree, the stable range exists, and furthermore coincides with the presentation degree.
\end{remark}

\subsection{Organization}

This paper is organized as follows. In Section 2 we give some background and preliminary knowledge on topos theory for the convenience of the reader. Sheaves of modules over categories equipped with the atomic topology are studied in Section 3, where we setup the translation machinery between sheaf theory and representation theory of categories, and prove all results listed in Theorem 1.1. In Section 4 we consider three types of categories with suitable combinatorial structure, describe equivalences between $\mathrm{Sh} (\BC, R)$ and some other representation categories which are more transparent compared to the Serre quotient category, and relate sheaf cohomology to certain local cohomology. In particular, we introduce the Nakayama functor, and prove Theorem \ref{second main theorem}. Discrete representations of topological groups are studied in Section 5, where we describe the classical result of Artin and consider some examples as an illustration of our approach, complete the classifications of irreducible sheaves or discrete representations listed in Theorem \ref{classfication}, and establish Theorem \ref{stability of discrete representations}.

\subsection{Notations}

Throughout this paper we stick to the following conventions unless otherwise stated: all functors are covariant functors, composition of maps or functors is from right to left (for example, the composite of a map $f: x \to y$ and a map $g: y \to z$ is written as $gf$), and actions of groups (resp., rings) on sets (resp., modules) are on the left side.

Some notations used throughout this paper are listed below:
\begin{itemize}
\item $\C$: a skeletally small category;
\item $\C^{\op}$: the opposite category of $\C$;
\item $J$: a Grothendieck topology on $\C$;
\item $J_{at}$: the atomic topology on $\C$, which satisfies the right Ore condition;
\item $\boldsymbol{\C} = (\C, J)$: a site (here $J$ is almost always $J_{at}$);
\item $\PSh(\C)$: the category of presheaves of sets over $\C$;
\item $\Sh(\boldsymbol{\C})$: the category of sheaves of sets over the site $\boldsymbol{\C}$;
\item $R$: a commutative ring;
\item $R\C$: the category algebra;
\item $\RR$: the constant presheaf of rings over $\C$ sending each object in $\C$ to $R$;
\item $P(x)$: the representable functor $R\C(x, -)$;
\item $\PSh(\C, R) = \C^{\op} \Mod$: the category of presheaves of $R$-modules over $\C$, or equivalently, the category of $\C^{\op}$-modules;
\item $\Sh(\BC, R)=\RR\Mod$: the category of sheaves of $R$-modules over the ringed site $(\BC, \RR)$ (equipped with the atomic topology), which are often called the $\RR$-modules;
\item $\mathrm{psh}(\C, R) = \C^{\op} \module$: the category of finitely generated $\C^{\op}$-modules;
\item $\mathrm{sh}(\BC, R)$: the category of finitely generated $\C^{\op}$-modules which are also sheaves of $R$-modules over the ringed site $(\BC, \RR)$ (equipped with the atomic topology);
\item $R \Mod$: the category of $R$-modules;
\item $\C \Mod^{\sat}$: the category of saturated $\C$-modules;
\item $\C \Mod^{\tor}$: the category of torsion $\C$-modules;
\item $\C \Mod^{\mathrm{tf}}$: the category of torsion free $\C$-modules;
\item $\C \module$: the category of finitely generated $\C$-modules;
\item $\C \module^{\sat}$: the category of finitely generated saturated $\C$-modules;
\item $\C \module^{\tor}$: the category of finitely generated torsion $\C$-modules;
\item $\C \fdmod$: the category of finite dimensional $\C$-modules;
\item $RG \DMod$: the category of discrete representations of a topological group $G$;
\item $\mathcal{P}_n$: the set of partitions of $[n]$;
\item $\mathcal{P}$: the disjoint union of $\mathcal{P}_n$ for all $n \in \N$.
\end{itemize}

\section{Preliminaries on topos theory}

For the convenience of the reader, in this section we introduce some basic concepts and results on topos theory. Good references for this area include \cite{Jo}, \cite{MM} and \cite{KW}.

\subsection{Grothendieck topology and sites}

Throughout this paper we let $\C$ be a \textit{skeletally small} category; that is, the isomorphism classes of objects form a set, and for every pair of objects $x$ and $y$, the class $\C(x, y)$ of morphisms from $x$ to $y$ is also a set. For an object $x$, a \textit{sieve} on $x$ is a subfunctor of the representable functor $\C(-, x)$; or equivalently, a \textit{right ideal} $S$ of $\C$ consisting of morphisms ending at $x$ in the following sense: if $f: y \to x$ is a member in $S$, and $g: z \to y$ is a morphism, then $fg = f \circ g \in S$ as well.

We recall the following definition of Grothendieck topology \cite[III.2, Definition 1]{MM}.

\begin{definition}
A \textit{Grothendieck topology} on $\C$ is a rule $J$ assigning to each object $x$ a family $J(x)$ of sieves on $x$ such that the following axioms are satisfied:
\begin{enumerate}
\item for each object $x$, $\C(-, x) \in J(x)$;
\item if $S \in J(x)$, then for any morphism $f: y \to x$, $f^{\ast} (S) = \{ g: \bullet \to y \mid fg \in S \}$ is contained in $J(y)$;
\item if $S \in J(x)$ and $R$ is any sieve on $x$ such that $f^{\ast}(R) \in J(y)$ for any $f: y \to x$ contained in $S$, then $R \in J(y)$.
\end{enumerate}
The pair $\boldsymbol{\C}=(\C, J)$ is called a \textit{site}, and members in $J(x)$ are called \textit{covering sieves} of $x$.
\end{definition}

{We shall point out that in the literature there are different versions of definitions of Grothendieck topologies. In this paper we always follow the definitions and conventions of \cite{MM}, and do not require categories to have fibre products. We also note that the assumption that $\C$ is skeletally small is essential since the first condition in the above definition requires $\C(y, x)$ to be a set.

There is a special type of Grothendieck topology called the \textit{trivial topology}: for every object $x$ one has $J(x) = \{ \C(-, x)\}$. This is the coarsest topology we may impose on $\C$.

\subsection{Sheaves of sets}

Let $\Sets$ be the category of sets and maps. A \textit{presheaf} of sets over $\C$ is a contravariant functor $F: \C \to \Sets$, or equivalently, a covariant functor from $\C^{\op}$ to $\Sets$. Morphisms between two presheaves of sets are natural transformations. Note that the definition of presheaves of sets is independent of Grothendieck topologies on $\C$. We denote the category of presheaves of sets by $\PSh(\C)$.

Now we fix a Grothendieck topology $J$ on $\C$ and let $\boldsymbol{\C} = (\C, J)$ be the corresponding site. The following definitions are taken from \cite[III.4]{MM}.

\begin{definition}
Let $F$ be a presheaf of sets and $S$ a covering sieve of a certain object $x$ in $\C$. A \textit{matching family} for $S$ of elements of $F$ is a rule assigning to each morphism $f: y \to x$ in $S$ an element $v \in F(y)$ such that $v_f \cdot g = v_{fg}$ for any morphism $g: \bullet \to y$, where the action of $g$ on $v_f$ is determined by the contravariant functor $F: \C \to \Sets$. An amalgamation of this matching family is an element $v_x \in F(x)$ such that $v_x \cdot f = v_f$ for all $f \in S$ \footnote{To respect our previous convention specified in Introduction, for actions of morphisms on sets via contravariant functors, we write the action on the right side.}.
\end{definition}

\begin{definition}
A presheaf $F \in \PSh(\C)$ is \textit{separated} if every matching family for every object $x$ and every covering sieve $S \in J(x)$ has at most one amalgamation. It is a \textit{sheaf} on the site $\boldsymbol{\C} = (\C, J)$ if every matching family has exactly one amalgamation.
\end{definition}

In particular, when $J$ is the trivial topology, every presheaf is a sheaf. In general every sheaf is a separated presheaf, and morphisms between two sheaves are morphisms between them as presheaves. Consequently, the category of sheaves over $\BC$ is a full subcategory of $\PSh(\C)$, and we denote it by $\Sh(\BC)$ (or simply $\Sh(\C)$ when the Grothendieck topology $J$ is clear from the context). There is a natural inclusion functor $\iota: \Sh(\BC) \to \PSh(\C)$, and it has a left adjoint functor $\sharp: \PSh(\C) \to \Sh(\BC)$, called the \textit{sheafification} functor. It is well-known that the sheafification is a special type of localization; see \cite[Proposition 1.3]{GZ} or \cite{KS}. It means that $\Sh(\BC)$ is a localization of $\PSh(\C)$ by the multiplicative system of local isomorphisms \cite{KS} (or morphismes bicouvrants \cite{AGV}).

\subsection{Atomic topology} \label{atomic topology}

In this paper we mostly consider a specific Grothendieck topology, called the atomic topology. To introduce its description, we need the following definition.

\begin{definition}
We say that $\C$ satisfies the \textit{left Ore condition} if for every pair of morphisms $f: x \to y$ and $f': x \to y'$, there exist an object $z$ and a pair of morphisms $g: y \to z$ and $g': y' \to z$ such that the following diagram commutes:
\[
\xymatrix{
x \ar[r]^-f \ar[d]_{f'} & y \ar@{-->}[d]^-g \\
y' \ar@{-->}[r]_-{g'} & z.
}
\]
The \textit{right Ore condition} is defined dually.
\end{definition}

For every object $x$ in $\C$, we let $J_{at}(x)$ be the collection of all nonempty sieves of $\C(-, x)$. The reader can check that the rule $J_{at}$ gives a Grothendieck topology if and only if $\C$ satisfies the right Ore condition; see \cite[III.2, Example (f)]{MM}. We call this topology the \textit{atomic topology}.

The following lemma gives a convenient criterion to determine whether a given presheaf over a category is a sheaf with respect to the atomic topology.

\begin{lemma}\cite[III.4, Lemma 2]{MM} \label{simple criterion}
Let $\C$ be a category with the atomic topology and $F$ a presheaf of sets over $\C$ . Then $F$ is a sheaf of sets on $\BC$ if and only if for any morphism $f: y \to x$, any morphisms $g$ and $h$ from a fixed domain to $y$, and any $v \in F(y)$ such that $v \cdot g = v \cdot h$ whenever $fg = fh$, then there exists a unique $u \in F(x)$ such that $u \cdot f = v$.
\end{lemma}

\begin{remark}\label{rmk 2.6} \normalfont
When replacing the words ``a unique" in the last sentence of this lemma by ``at most one", we obtain a criterion for a presheaf of sets to be separated.
\end{remark}

Recall that a presheaf $F$ is \textit{constant} if $F(x) = F(y)$ for any objects $x$ and $y$, and any morphism $f: y \to x$ acts (on the right side) as the identity map $F(y) \to F(x)$. It is easy to see that the constant presheaf is separated since every morphism acts as an injective map, and furthermore, by the previous lemma, it is actually a sheaf for the atomic topology.

\subsection{Sheaves of modules} \label{sheaves of modules}

To consider applications of sheaf theory in representation theory, we need to study sheaves with extra algebraic structures, in particular, ringed sites and sheaves of modules. In this subsection we only give necessary definitions. For more details, please refer to \cite[Chapter 18]{Stack}.

\begin{definition}
Let $\boldsymbol{\C} = (\C, J)$ be a site. A \textit{presheaf of abelian groups} is a contravariant functor $\mathcal{O}: \C \to \mathrm{Ab}$, the category of abelian groups; it is called a \textit{sheaf of abelian groups} if the underlying presheaf of sets is a sheaf. Abelian presheaves (resp. sheaves) form abelian categories $\PSh(\C, \mathbb{Z})=\PSh(\C, \mathrm{Ab})$ (resp. $\Sh(\BC, \mathbb{Z})=\Sh(\BC, \mathrm{Ab})$).

More generally we have the category of (pre)sheaves of $R$-modules, and it is an abelian category $\PSh(\C, R) = \PSh(\C, R\Mod)$ (or $\Sh(\BC, R) = \Sh(\BC, R\Mod))$.

Similarly, one can define presheaves and sheaves of rings.
\end{definition}

\begin{definition}
Let $\mathcal{O}$ be a sheaf of rings. The pair $(\BC, \mathcal{O})$ is called a \textit{ringed site}, and $\mathcal{O}$ is called the \textit{structure sheaf} of the ringed site.
\end{definition}

\begin{definition}
Let $(\BC, \mathcal{O})$ be a ringed site. A \textit{presheaf of $\mathcal{O}$-modules} is a presheaf $F$ of abelian groups and a map of presheaf of sets
\[
\mathcal{O} \times F \to F
\]
such that for every object $x$, the map
\[
\mathcal{O}(x) \times F(x) \to F(x)
\]
defines an $\mathcal{O}(x)$-module structure of $F(x)$. The subcategory of $\PSh(\BC,\mathrm{Ab})$ consisting of them is denoted by $\PSh(\C, \mathcal{O})$.

We say that $F$ is a \textit{sheaf of $\mathcal{O}$-modules} if the underlying presheaf of sets is a sheaf. The subcategory of $\Sh(\BC,\mathrm{Ab})$ that they form is denoted by $\mathcal{O}\Mod$.
\end{definition}

Note that $\PSh(\C, \mathcal{O})$ is an abelian category, and kernels and cokernels are computed in the underlying category $\PSh(\C, \mathrm{Ab})$. The category $\mathcal{O} \Mod$ is also abelian: kernels are computed in $\PSh(\C, \mathcal{O})$; to compute cokernels, we first obtain the cokernels in $\PSh(\C, \mathcal{O})$, and then apply the sheafification functor $\sharp$. Furthermore, these two abelian categories are Grothendieck categories, and hence every object in them has an injective hull.

When $\mathcal{O}=\RR$ is the constant sheaf (of rings), we have $\RR\Mod=\Sh(\BC,R)$.

\section{Sheaves of modules on atomic sites}

From now on we suppose that $\C^{\op}$ \footnote{The reason we consider sheaf theory of $\C^{\op}$ rather than that of $\C$ is that presheaves of $\C^{\op}$ are exactly $\C$-modules, and in the main body of this paper we will study representation theory of $\C$ instead of that of $\C^{\op}$.} is a skeletally small category satisfying the right Ore condition, and impose the atomic topology $J_{at}$ on $\C^{\op}$; that is, for any object $x$, $J_{at}(x)$ is the collection of subfunctors $S$ of $\C^{\op}(-, x) = \C(x, -)$ such that $\C(x, y) \neq \emptyset$ for at least one object $y$. We also fix $\boldsymbol{\C}^{\op}$ to be the pair $(\C^{\op}, J_{at})$.

Let $R$ be a commutative ring and $\RR$ be the constant presheaf of rings such that $\RR(x) = R$ for every object $x$ in $\C^{\op}$. By Lemma \ref{simple criterion}, $\RR$ is a sheaf of rings. Let $(\BC^{\op}, \RR)$ be the ringed site. The main goal of this section is to characterize sheaves of $\RR$-modules in terms of representations of $\C$, and describe the structure of $\Sh(\BC^{\op}, R)$.

Since $\RR$ is a constant sheaf of rings, an object in $\PSh(\C^{\op}, R)$ is nothing but a \textit{$\C$-module}; that is, a covariant functor from $\C$ to the category $R \Mod$. Consequently, objects in $\Sh(\BC^{\op}, R)$ are $\C$-modules such that the underlying presheaves of sets are actually sheaves. Thus we may freely use the language of representation theory of $\C$. In particular, we also denote $\PSh(\C^{\op}, R)$ by $\C \Mod$.

\subsection{Characterizations of sheaves} Now let $V$ be an object in $\PSh(\C^{\op}, R)$, or equivalently, a $\C$-module over the coefficient ring $R$. For an object $x$ in $\C$, we denote the value of $V$ on $x$ by $V_x$. An element $v \in V_x$ is \textit{torsion} if there is a morphism $f: x \to y$ in $\C$ such that $f \cdot v = 0$. Define a sub-presheaf $V_T$ of $V$ by setting
\[
V_T (x) = \{ v \in V_x \mid v \text{ is torsion } \}
\]
for each $x \in \Ob(\C)$. The following lemma shows that $V_T$ is indeed a submodule of $V$.

\begin{lemma}
$V_T$ is a submodule of $V$.
\end{lemma}

\begin{proof}
It is clear that $V_T(x)$ is closed under scalar multiplication by elements in $R$ for $x \in \Ob(\C)$. It is also closed under addition. Indeed, given $v, v' \in V_T(x)$, one may find $f: x \to y$ and $f': x \to y'$ such that $f \cdot v = 0 = f' \cdot v'$. Now by the left Ore condition, we may find another object $z$ and morphisms $g: y \to z$ and $g': y' \to z$ such that $gf = g'f'$. Then one has
\[
gf \cdot (v + v') = gf \cdot v + g'f' \cdot v' = 0,
\]
so $v + v'$ is also torsion.

It remains to prove the following statement: for any object $x$, any torsion element $v \in V_x$, and any morphism $f: x \to y$, $f \cdot v \in V_y$ is also a torsion element. Since $v$ is torsion, we can find a morphism $f_x: x \to z$ such that $f_x \cdot v = 0$. By the left Ore condition on $\C$, we obtain a commutative diagram:
\[
\xymatrix{
x \ar[r]^-{f_x} \ar[d]_f & z \ar[d]^-g \\
b \ar[r]_-{g'} & a.
}
\]
Consequently, $g' \cdot (f \cdot v_x) = (gf_x) \cdot v_x = g \cdot (f_x \cdot v) = 0$, so $f \cdot v_x$ is torsion as claimed.
\end{proof}

The left Ore condition is essential for us the impose the atomic topology, and is also essential for us to define the torsion theory, as shown in the following example.

\begin{example} \normalfont
Let $Q$ be the Kronecker quiver with two objects $x$ and $y$, and two arrows $\alpha$ and $\beta$ from $x$ to $y$. Consider $V = \mathbb{R}Q / \langle \beta \rangle$. Then the image of $1_x$ (the identity morphism on $x$) in the quotient $V$ (still denote it by $1_x$) is torsion, since $\beta$ maps it to 0. However, $\alpha \cdot 1_x = \alpha$ is not torsion. Thus the torsion elements do not form a submodule of $V$.
\end{example}

It is obvious that $V_T$ is the maximal torsion submodule of $V$. We say that $V$ is \textit{torsion free} if $V_T = 0$, and is \textit{torsion} if $V_T = V$. Moreover, one can easily check that the category $\C \Mod^{\tor}$ of torsion $\C$-modules is an abelian subcategory of $\C \Mod$, and the assignment $V \to V_T$ gives rise to a left exact functor $\tau$, called the \textit{torsion functor}. The category $\C \Mod^{\mathrm{tf}}$ consisting of torsion free $\C$-modules is in general not abelian. However, these two subcategories form a \textit{torsion pair} in the following sense:

\begin{lemma} \label{torsion pair}
Notation as above. Then one has:
\begin{enumerate}
\item $V/V_T$ is torsion free;
\item if $V$ is a torsion $\C$-module and $W$ is torsion free $\C$-module, then $\Hom_{\C \Mod} (V, W) = 0$;
\item $\Hom_{\C \Mod} (V, W) = 0$ for all $W \in \C \Mod^{\mathrm{tf}}$ if and only if $V \in \C \Mod^{\tor}$;
\item dually, $\Hom_{\C \Mod} (V, W) = 0$ for all $V \in \C \Mod^{\tor}$ if and only if $W \in \C \Mod^{\mathrm{tf}}$.
\end{enumerate}
\end{lemma}

\begin{proof}
If $V/V_T$ is not torsion free, then by definition $(V/V_T)_T \neq 0$. In particular, one can find an object $x$, a nonzero element $\bar{v} \in (V/V_T)_x$, and a morphism $f: x \to y$ such that $f \cdot \bar{v} = 0$. Take a preimage $v \in V_x$ of $\bar{v}$. Then $f \cdot v \in (V_T)_y$ is torsion, and hence one can find another morphism $g: y \to z$ such that $g \cdot (f \cdot v) = 0$. In particular, $v$ is torsion as well, so $\bar{v} = 0$. This contradiction tells us that $V/V_T$ is torsion free.

Suppose that there is a nonzero homomorphism $\varphi: V \to W$. Then in particular there exist an object $x$ and a nonzero element $v \in V_x$ such that $\varphi_x(v) \neq 0$. But since $V$ is torsion, one can find a morphism $f: x \to y$ such that $f \cdot v = 0$. Consequently, $f \cdot \varphi_x(v) = \varphi_y (f \cdot v) = 0$. But $\varphi_x(v) \in W_x$, so $W$ is not torsion free, and hence (2) follows from this contradiction.

The ``if" directions of (3) and (4) are clear. For the other directions, take $W$ (resp., $V$) to be $V/V_T$ (resp., $W_T$). Then $\Hom_{\C \Mod} (V, V/V_T) = 0$ or $\Hom_{\C \Mod} (W_T, W) = 0$ implies that $V/V_T = 0$ or $W_T = 0$.
\end{proof}

The following lemma interprets separated presheaves as torsion free modules.

\begin{lemma} \label{separated presheaves}
Separated presheaves in $\PSh(\C^{\op}, R)$ are precisely torsion free $\C$-modules.
\end{lemma}

\begin{proof}
Suppose that $V$ is a torsion free module. Note that $V$ is torsion free if and only if for every morphism $f: x \to y$ in $\C$, the corresponding $R$-linear homomorphism $V_f: V_x \to V_y$ is injective. By Remark \ref{rmk 2.6} (note that we have to reverse all arrows since we are working with $\C^{\op}$ rather than $\C$), $V$ is a separated presheaf of $R$-modules. Conversely, if $V$ is not torsion free, then there exists a morphism $f: x \to y$ such that $V_f: V_x \to V_y$ is not injective. In particular, if we let $g = h = 1_y$ and let $v = 0$  in Lemma \ref{simple criterion}, then one can find two distinct elements (one of which is 0) in $V_x$ such that both are sent to $v = 0 \in V_y$ by $V_f$, and hence $V$ is not a separated presheaf.
\end{proof}

Our next task is to characterize sheaves. For this purpose, we introduce a few more notions.

The functor $\tau: \C \Mod \to \C \Mod^{\tor}$ defined by $V \mapsto V_T$ is left exact, and hence have right derived functors $\mathrm{R}^i \tau$, $i \in \N$. To compute $\mathrm{R}^i \tau (V)$ for a $\C$-module $V$, we only need to apply the functor $\tau$ to an injective resolution of $V$, and cohomology groups of this complex give $\mathrm{R}^i \tau(V)$ for $i \geqslant 0$. Clearly, $V$ is torsion free if and only if $\mathrm{R}^0 \tau (V) = \tau V = V_T = 0$.

\begin{definition} \label{saturated modules}
A $\C$-module $V$ is said to be \textit{saturated} if $\mathrm{R}^i \tau (V) = 0$ for $i = 0, 1$.
\end{definition}

The next lemma tells us that saturated modules are precisely modules right perpendicular to torsion modules in the sense of Geigle and Lenzing (see the beginning of \cite[Section 1]{GLen}). Thus the above definition is equivalent to the one used in \cite[Subsection 4.1]{SS}, where a module is called saturated if it is torsion free and has no nontrivial extension by torsion modules.

\begin{lemma} \label{saturated modules}
A $\C$-module $V$ is saturated if and only if $V$ is right perpendicular to all torsion modules.
\end{lemma}

\begin{proof}
The conclusion follows from a straightforward homological argument. For the convenience of the reader we provide a sketch of the proof. Clearly, $V$ is torsion free if and only if $\Hom_{\C \Mod}(T, V)$ vanishes for any torsion module $T$. Thus it suffices to show the conclusion for torsion free modules. Suppose that $V$ is torsion free. Since $\C \Mod$ is a Grothendieck category, $V$ has an injective hull $E$ in $\C \Mod$. We claim that $E$ is torsion free. Otherwise, it has a nonzero torsion submodule $E_T$. In this case, the intersection $V \cap E_T$ must be 0 since $V$ is torsion free and $E_T$ is torsion. This contradicts the fact that the inclusion $V \to E$ is essential by \cite[Definition 47.2.1]{Stack}.

Note that $E$ is torsion free and injective. Applying the functor $\Hom_{\C \Mod} (T, -)$ to the short exact sequence $0 \to V \to E \to E/V \to 0$ we deduce that $V$ is right perpendicular to every torsion module $T$ if and only if $E/V$ is torsion free, which is equivalent to saying that $V$ is saturated via applying the functor $\tau$ to the short exact sequence.
\end{proof}

Now we are ready to prove the main result of this subsection.

\begin{theorem} \label{main result 1}
A presheaf in $\PSh(\C^{\op}, R)$ is a sheaf of $R$-modules if and only if it is saturated as a $\C$-module.
\end{theorem}

\begin{proof}
Suppose that $V$ is an object in $\Sh(\BC^{\op}, R)$. In particular, it is a separated presheaf. By Lemma \ref{separated presheaves}, $V$ is torsion free. Since the category of torsion modules is a localizing subcategory by \cite[Proposition 2.5]{GLen}, applying \cite[Proposition 2.2]{GLen} we obtain a short exact sequence
\[
0 \to V \to \tilde{V} \to W \to 0
\]
such that $W$ is torsion and $\tilde{V}$ is right perpendicular to all torsion modules, and hence saturated by Lemma \ref{saturated modules}. To show that $V$ is saturated, it suffices to check that $W$ is torsion free, which forces $W = 0$ and hence $V \cong \tilde{V}$.

Suppose that there exist an object $x$, an element $w \in W_x$, and a morphism $f: x \to y$ in $\C$ such that $f \cdot w = 0$. Let $\tilde{w} \in \tilde{V}_x$ be a preimage of $w$ and consider the left ideal $S = \{gf \mid g: y \to \bullet \}$ of $\C$ generated by $f$ (which is a covering sieve of $x$ in $\C^{\op}$). Since $gf \cdot w = 0$, one deduces that $gf \cdot \tilde{w} \in V_{\bullet}$. Thus $\{g f \cdot \tilde{w} \mid g: y \to \bullet \}$ is a matching family for $S$ of elements of $V$. Since $V$ is a sheaf, one can find an amalgamation $\bar{w} \in V_x \subseteq \tilde{V}_x$ such that $ gf \cdot \bar{w} = gf \cdot \tilde{w}$ for every $g$ with domain $y$. In particular, by taking $g = 1_y$ one shall have $f \cdot \tilde{w} = f \cdot \bar{w}$. But $\tilde{V}$ is torsion free, so $\tilde{w} = \bar{w} \in V_x$, and hence $w$ as the image of $\tilde{w}$ is 0. Consequently, $W$ is torsion free as desired.

Conversely, suppose that $V$ is saturated as a $\C$-module, and in particular it is torsion free. Therefore, $V$ is a separated presheaf in $\PSh(\C^{\op}, R)$. By \cite[III.5, Lemma 2]{MM}, the canonical map $\eta: V \to V^+$ is injective where $V^+$ is the sheafification of $V$, and by \cite[III.5, Lemma 5]{MM}, $V^+$ is a sheaf in $\Sh(\BC^{\op}, R)$. Therefore, one obtains the following exact sequence of $\C$-modules
\[
0 \to V \to V^+ \to W \to 0.
\]
Note that $V^+$ is a saturated $\C$-module by the result we just proved. We show that $V$ is a sheaf by contradiction.

If $V$ is a separated presheaf in $\PSh(\C^{\op}, R)$ but is not a sheaf in $\Sh(\BC^{\op}, R)$, then there exist an object $x$, a nonempty covering sieve $S \subseteq \C(x, -) = \C^{\op}(-, x)$ of $x$, and a matching family $ \{v_f \in V_y \mid (f: x \to y) \in S \}$ for $S$ of elements of $V$ satisfying $g \cdot v_f = v_{gf}$ for any morphism $g: y \to \bullet$, such that there does not exist an element $v \in V_x$ satisfying $f \cdot v = v_f$ for all $f \in S$. Now identify $V$ with a subpresheaf of $V^+$ via the canonical map $\eta$. The above matching family of elements of $V$ becomes a matching family of elements of $V^+$. Since $V^+$ is a sheaf, we can find an element $\tilde{v} \in V_x^+$ serving as the unique amalgamation of this matching family of elements of $V^+$. Clearly, $\tilde{v}$ is not contained in $V_x$ since otherwise $\tilde{v}$ as an element in $V_x$ can serve as the unique amalgamation of the matching family of elements of $V$. Consequently, the image $\bar{v} \in W_x$ of $\tilde{v}$ is nonzero. However, for every $f: x \to \bullet$ in $S$, one has $f \cdot \tilde{v} = v_f \in V_{\bullet}$, and hence in the quotient module $W$ one must have $f \cdot \bar{v} = 0$. That is, $W$ not torsion free. Applying $\tau$ to the above short exact sequence we have:
\[
0 \to \tau V = 0 \to \tau V^+ = 0 \to \tau W \to \mathrm{R}^1 \tau (V) \to \mathrm{R}^1 \tau (V^+) = 0 \to \ldots.
\]
Consequently, $\mathrm{R}^1 \tau (V) \cong \tau W \neq 0$, and hence $V$ is not saturated. This contradiction tells us that $V$ must be a sheaf.
\end{proof}

We obtain the following immediate corollaries.

\begin{corollary} \label{injective sheaves}
Injective objects in $\Sh(\BC^{\op}, R)$ are precisely torsion free injective $\C$-modules.
\end{corollary}

\begin{proof}
Theorem \ref{main result 1} asserts that torsion free injective $\C$-modules $I$ are contained in $\Sh(\BC^{\op}, R)$. They are injective objects in this category. Indeed, given a momomorphism $0 \to I \to V$ in $\Sh(\BC^{\op}, R)$, it is also a monomorphism in $\PSh(\C^{\op}, R)$, and hence splits in the category of presheaves. Since sheafification does not change $V$, $I$ and the morphism, and every functor takes retractions to retractions, it turns out that the above monomorphism actually splits in $\Sh(\BC^{\op}, R)$. Thus $I$ is an injective sheaf. Furthermore, given any $V \in \Sh(\BC^{\op}, R)$, one can obtain a monomorphism $0 \to V \to I$ in $\PSh(\C^{\op}, R)$ such that $I$ is a torsion free injective $\C$-module, which is also a monomorphism in $\Sh(\BC^{\op}, R)$. Thus any sheaf can be embedded into a sheaf which is also a torsion free injective $\C$-module. Consequently, any injective sheaf must be a torsion free injective $\C$-module.
\end{proof}

Let $\C \Mod^{\sat}$ be the full subcategory of $\C \Mod$ consisting of saturated $\C$-modules.

\begin{corollary} \label{equivalences}
One has the following identification and equivalence of categories:
\[
\xymatrix{
\Sh(\BC^{\op}, R) \ar[r]^-{=} & \C \Mod^{\sat} \ar[r]^-{\simeq} & \C \Mod / \C \Mod^{\tor}.
}
\]
\end{corollary}

\begin{proof}
The first identity is clear since both categories are full subcategories of $\C \Mod$ and have the same objects by Theorem \ref{main result 1}. The second equivalence follows immediately from \cite[Proposition 2.2(d)]{GLen}.
\end{proof}

\begin{remark} \normalfont
Although the category $\C \Mod^{\sat}$ is a full subcategory of $\C \Mod$, its abelian structure is not inherited from that of $\C \Mod$. This is an analogue of the following well known fact: the category of sheaves is a full subcategory of the category of presheaves, but they have different abelian structures.
\end{remark}

It is well known that $\Sh(\BC^{\op}, R)$ is a localization of $\PSh(\C^{\op}, R)$ by the Serre subcategory consisting of $\C$-modules $V$ with $V^{\sharp}=0$. In general, it is not an easy task to describe this localization since the definition of $\sharp$ is complicated, see for instance \cite[III.5]{MM}. For the atomic topology, one may get a very transparent description: one takes the torsion free part $V_F$ of $V$, embeds it into an injective hull $I$ (which is saturated) to get a short exact sequence $0 \to V_F \to I \to C \to 0$ of $\C$-modules, and then constructs the pull back of $I \to C$ and $C_T \to C$ to obtain the following commutative diagram such that all rows and columns are exact:
\[
\xymatrix{
0 \ar[r] & V_F \ar[r] \ar@{=}[d] & \widetilde{V} \ar[r] \ar[d] & C_T \ar[d] \ar[r] & 0\\
0 \ar[r] & V_F \ar[r] & I \ar[d] \ar[r] & C \ar[r] \ar[d] & 0\\
 & & C_F \ar@{=}[r] & C_F.
}
\]
Then $V^{\sharp} \cong \widetilde{V}$. Indeed, from the middle column we conclude that $\widetilde{V}$ is a saturated module, and then from the top row we deduce that $V^{\sharp} \cong V_F^{\sharp} \cong \tilde{V}^{\sharp} = \tilde{V}$.

In the rest of this subsection we describe another interpretation of $\Sh(\BC^{\op}, R)$ via the notion of morphism categories. Let $\mathfrak{I}$ be the additive category of torsion free injective $\C$-modules (which are precisely injective objects in $\Sh(\BC^{\op}, R)$). Following from Auslander, Reiten and Smal{\o} (see \cite[IV.1]{ARS}), we define the \textit{morphism category} $\Mor(\mathfrak{I})$ as follows:
\begin{enumerate}
\item[$\bullet$] objects are morphisms $\alpha: I^0 \to I^1$ in $\mathfrak{I}$;
\item[$\bullet$] morphisms from $\alpha: I^0 \to I^1$ to $\beta: J^0 \to J^1$ are pairs $(f^0,f^1)$ such that the following diagram commutes
\[
\xymatrix{
I^0 \ar[r]^-{\alpha} \ar[d]^-{f^0} & I^1 \ar[d]^-{f^1}\\
J^0 \ar[r]^-{\beta} & J^1.
}
\]
\end{enumerate}

We now define the functor $\ker: \Mor(\mathfrak{I}) \to \C \Mod^{\sat}$ as follows:
\begin{enumerate}
\item[$\bullet$] for each object $\alpha: I^0 \to I^1$ in $\Mor(\mathfrak{I})$, $\ker\alpha \in \C \Mod^{\sat}$ by Proposition \ref{saturated modules};
\item[$\bullet$] for each morphism $(f^0,f^1)$ in $\Mor(\mathfrak{I})$ from $\alpha: I^0 \to I^1$ to $\beta: J^0 \to J^1$, we let $\ker(f^0,f^1)$ be the unique morphism such that the diagram
    \[
\xymatrix{
0 \ar[r] & \ker\alpha \ar[d] \ar[r] & I^0 \ar[r]^\alpha \ar[d]^{f^0} & I^1 \ar[d]^{f^1}\\
 0 \ar[r] & \ker\beta \ar[r] & J^0  \ar[r]^\beta & J^1 \\
}
\]
is commutative.
\end{enumerate}
It is straightforward to see that $\ker$ is a full and dense functor, but not faithful. However, the reader can easily check that $\ker(f^0,f^1)=0$ if and only if there is a morphism $f: I^1 \to J^0$ such that $f\alpha=f^0$. Therefore, to obtain an equivalence between categories, we need to modulo some relations.

For two objects $\alpha: I^0 \to I^1$ and $\beta: J^0 \to J^1$ in $\Mor(\mathfrak{I})$, we let $\mathfrak{R}(\alpha,\beta)$ be the subgroup of $\Hom_{\Mor(\mathfrak{I})}(\alpha,\beta)$ consisting of those morphisms $(f^0,f^1)$ such that there exists a morphism $f: I^1 \to J^0$ with $f\alpha=f^0$. Then $\mathfrak{R}$ gives a relation on $\Mor(\mathfrak{I})$ in the sense of \cite{ARS}. As Auslander, Reiten and Smal{\o} proved in \cite[Proposition IV.1.2]{ARS}, we have the following result.

\begin{proposition}
The functor $\ker: \Mor(\mathfrak{I}) \to \C \Mod^{\sat}$ induces an equivalence between $\Mor(\mathfrak{I})/\mathfrak{R}$ and $\C \Mod^{\sat}$.
\end{proposition}

In summary, we obtain the following equivalences of categories:
\[
\xymatrix{
\Sh(\BC^{\op}, R) \ar[r]^-{=} & \C \Mod^{\sat} \ar[r]^-{\simeq} & \C \Mod / \C \Mod^{\tor} \ar[r]^-{\simeq} & \Mor(\mathfrak{I})/\mathfrak{R}.
}
\]

\subsection{Sheaf cohomology}

In this subsection we illustrate a relation between sheaf cohomology functors and derived functors $\mathrm{R}^{\bullet} \tau$ of the torsion functor $\tau$.

Given a presheaf $V \in \PSh(\C^{\op}, R)$, we can construct an injective resolution $0 \to V \to I^{\bullet}$ of presheaves. Now for every object $x \in \C^{\op}$, one defines a functor $\Gamma^p_x: \PSh(\C^{\op}, R) \to R \Mod$ by $V \mapsto V_x$, called an \textit{evaluation} functor, which is exact. Similarly, one can define an evaluation functor $\Gamma_x: \Sh(\BC^{\op}, R) \to R \Mod$, which is only left exact, so one can define $\mathrm{R}^i \Gamma_x(V)$, the $i$-th right derived functor of $\Gamma_x$.

\begin{remark} \normalfont
Note that $\mathrm{R}^i \Gamma_x(V)$ is called the $i$-th \textit{sheaf cohomology group} of $V$ at $x$. For details, please refer to \cite[I.21.2]{Stack}.
\end{remark}

To compute $\mathrm{R}^{\bullet} \Gamma_x(V)$, one constructs an injective resolution $0 \to V \to I^{\bullet}$ of sheaves and apply $\Gamma_x$ to obtain a complex $0 \to \Gamma_x V \to \Gamma_x I^{\bullet}$. Cohomology groups of this complex give $\mathrm{R}^{\bullet} \Gamma_x(V)$. The following result relate $\mathrm{R}^{\bullet} \Gamma_x(V)$ to right derived functors of $\tau$.

\begin{theorem} \label{sheaf cohomology}
For $V \in \Sh(\BC^{\op}, R)$ and object $x$ in $\C^{\op}$, one has $\mathrm{R}^i \Gamma_x(V) \cong (\mathrm{R}^{i+1} \tau (V))_x$ for $i \geqslant 1$, the value of the $\C$-module $\mathrm{R}^{i+1} \tau (V)$ on the object $x$.
\end{theorem}

Note that $\mathrm{R}^{i+1} \tau (V)$ is also a $\C$-module, and hence its value on $x$ is well defined.

\begin{proof}
Firstly we take a short exact sequence $0 \to V \to I^0 \to V^1 \to 0$ in $\C \Mod$ such that $I^0$ is a torsion free injective $\C$-module. Note that this is not an exact sequence in $\Sh(\BC^{\op}, R)$ since $V^1$ might not even be a sheaf, although it is a separated presheaf. Thus we apply the functor $\sharp$ to get the following short exact sequence in $\Sh(\BC^{\op}, R)$:
\[
0 \to V^{\sharp} = V \to (I^0)^{\sharp} = I^0 \to (V^1)^{\sharp} \to 0.
\]

 By our description of sheafification functor in subsection 3.1, we have the following commutative diagram of short exact sequences of $\C$-modules:
\[
\xymatrix{
0 \ar[r] & V^1 \ar[r] \ar@{=}[d] & (V^1)^{\sharp} \ar[r] \ar[d] & \tau V^2 \ar[d] \ar[r] & 0\\
0 \ar[r] & V^1 \ar[r] & I^1 \ar[r] & V^2 \ar[r] & 0.
}
\]

Now applying the evaluation functors $\Gamma_x^p$ and $\Gamma_x$, we obtain the following exact sequences
\begin{align*}
& 0 \to \Gamma^p_x V \to \Gamma^p_x I^0 \to \Gamma_x^p V^1 \to 0\\
& 0 \to \Gamma_x V \to \Gamma_x I^0 \to \Gamma_x ((V^1)^{\sharp}) \to \mathrm{R}^1 \Gamma_x(V) \to 0\\
& 0 \to \Gamma^p_x V^1 \to \Gamma^p_x ((V^1)^{\sharp}) \to \Gamma_x^p (\tau V^2) \to 0.
\end{align*}
Note that $\Gamma^p_x F = \Gamma_x F$ whenever $F \in \Sh(\BC^{\op}, R)$. Therefore, by comparing these sequences, one gets
\[
\mathrm{R}^1 \Gamma_x(V) \cong \Gamma_x^p (\tau V^2) \cong (\tau V^2)_x.
\]
However, it is straightforward to check that $\tau V^2 \cong \mathrm{R}^2 \tau (V)$, so the conclusion holds for $i = 1$. Now replacing $V$ by $(V^1)^{\sharp}$ we deduce the isomorphism for $i = 2$. The conclusion follows from an induction.
\end{proof}

\subsection{Finiteness conditions}

For future applications in the upcoming Section 4, we end this section with some remarks on finiteness conditions on presheaves and sheaves, which was dealt by Auslander for functor categories in \cite[Section 1]{Au}. We follow the terminologies from \cite[Section 3.5]{Po}.

\begin{proposition} \label{finite type}
Let $R$ be a commutative ring. Then a presheaf $V \in \PSh(\C^{\op}, R)$ is of finite type if and only if the $\C$-module $V$ is finitely generated.
\end{proposition}

\begin{proof}
This follows from \cite[Chapter 3, Corollary 5.7]{Po} and the remark immediately after this corollary, where the small preaddtive category there can be taken as the $R$-linearization of the category $\C$ considered by us in this paper.
\end{proof}

We denote the category of presheaves of $R$-modules of finite type by $\mathrm{psh}(\C^{\op}, R)$ or $\C \module$, and denote by $\mathrm{sh}(\BC^{\op},R)$ the full subcategory of $\Sh(\BC^{\op}, R)$ consisting of sheaves which, considered as presheaves, are of finite type. Clearly, when the category algebra $R\C$ is locally Noetherian, $\C \module$ is an abelian subcategory of $\C \Mod$, and $\mathrm{sh}(\BC^{\op},R)$ is an abelian subcategory of $\Sh(\BC^{\op}, R)$.

\begin{remark}
We must warn the reader that presheaves of finite type and sheaves over a ringed site of finite type are quite different. Indeed, given a ringed site $(\BC^{\op},\mathcal{O})$, the characterization of sheaves with finiteness conditions seems to be more subtle, which heavily relies on the structure sheaves on $\BC^{\op}$. For instance if we take $\mathcal{O}$ to be the constant sheaf of rings $\RR$, then $\RR \Mod = \Sh(\BC^{\op},R)$. In order to recognize local types of modules, see \cite[Sec 18.1]{KS} or \cite[Sec 18.23]{Stack}, one may rely on the definition and use various localization morphisms of sites $\BC^{\op}/x \to \BC^{\op}$, $x \in \Ob(\C)$, where $\C^{\op}/x$ is the category over $x$ with induced topology.

For example, let $\C$ be the skeletal subcategory of $\FI$ consisting of objects $[n]$, $n \in \N$, and put the atomic topology on $\C^{\op}$. It is easy to see that $P([n]) = R\C([n], -) \in \Sh(\BC^{\op}, R)$ for any $n \in \N$. However, for $n \geqslant 1$, $P([n])$ as a presheaf is of finite type by Proposition \ref{finite type}, but as a sheaf it is not of finite type in the sense of \cite[Section 18.23]{Stack}, since the rank of $P([n])_m = R \C([n],[m])$ as a free $R$-module strictly increases when $m$ does.
\end{remark}

\section{Applications to combinatorial categories}

In the previous section we have shown the equivalence between $\Sh(\BC^{\op}, R)$ and $\C \Mod / \C \Mod^{\tor}$. Since this quotient category is still mysterious for practical purposes, in this section we give more transparent descriptions for some special types of combinatorial categories $\C$. Throughout this section we assume that $\C$ satisfies the left Ore condition and equip $\C^{\op}$ with the atomic topology.

\subsection{Categories with enough maximal objects} \label{type i cats}

In this subsection we consider small skeletal categories $\C$ satisfying the following conditions (called \textit{Type I combinatorial categories}):
\begin{itemize}
\item $\C$ is \textit{directed}: the relation that $x \preccurlyeq y$ if and only if $\C(x, y) \neq \emptyset$ is a partial ordering on the set of objects in $\C$;
\item $\C$ \textit{has enough maximal objects}: for every object $x$ in $\C$, there is another object $y$ such that $x \preccurlyeq y$ and $y$ is maximal with respect to $\preccurlyeq$.
\item if $y$ is maximal with respect to $\preccurlyeq$, then $\C(y, y)$ is a group.
\end{itemize}

Examples of Type I combinatorial categories include: skeletal subcategories of the orbit category of a finite group $G$ consisting of cosets $G/H$ (where $H$ is a subgroup) and $G$-equivariant maps, skeletal subcategories of the category of finite sets and surjections (with the singleton set as the unique maximal element), the poset of open sets with respect to inclusions of a given topological space $X$ (with $X$ as the unique maximal element).

Now let $S$ be the set of maximal objects in $\C$. For $x \in S$, we denote $\C(x, x)$ by $G_x$. The main result of this subsection is:

\begin{theorem}
Suppose that $\C$ is a Type I combinatorial category. Then one has the following equivalence:
\[
\Sh(\BC^{\op}, R) \simeq \prod_{x \in S} RG_x \Mod,
\]
where $\prod_{x \in S} RG_x \Mod$ is the product category of $RG_x \Mod$ indexed by $x \in S$. Explicitly, objects in $\prod_{x \in S} RG_x \Mod$ are families $(V_x)_{x \in S}$ such that $V_x$ is an $RG_x$-module.
\end{theorem}

\begin{proof}
Since we have shown that $\Sh(\BC^{\op}, R)$ is equivalent to $\C \Mod/ \C \Mod^{\tor}$, and it is also easy to see that $\prod_{x \in S} RG_x \Mod \cong \D \Mod$, where $\D$ is the full subcategory of $\C$ whose objects are precisely elements in $S$, one only needs to show that $\C \Mod/ \C \Mod^{\tor} \simeq \D \Mod$. The proof of this conclusion is based on \cite[Proposition III.2.5]{Gab}.

The inclusion functor $\D \to \C$ induces a restriction functor $\textsl{res}: \C \Mod \to \D \Mod$, which is exact. It is well known that $\textsl{res}$ has a right adjoint functor $\textsl{coind}: \D \Mod \to \C \Mod$, called the \textit{coinduction} functor, which is a special case of the right Kan extension. Explicitly, given a $\D$-module $V$, $\textsl{coind}(V)$ is defined as follows: for every object $x$ in $\C$, one has
\[
(\textsl{coind}(V))_x = \Hom_{\D \Mod} (R\C(x, -), V).
\]
Furthermore, it is direct to check that $\textsl{res} \circ \textsl{coind}: \D \Mod \to \D \Mod$ is isomorphic to the identity functor on $\D \Mod$, using the fact that objects in $\D$ are maximal objects in $\C$ with respect to $\preccurlyeq$. Therefore, by \cite[Proposition III.2.5]{Gab}, we obtain the following equivalence:
\[
\C \Mod / \mathrm{ker} (\textsl{res}) \simeq \D \Mod,
\]
where $\mathrm{ker} (\textsl{res})$ is the full subcategory of $\C \Mod$ consisting of $\C$-modules $V$ such that $\textsl{res} (V) = 0$. It remains to show that $\mathrm{ker} (\textsl{res}) = \C \Mod^{\tor}$, which is done in the next lemma.
\end{proof}

\begin{lemma}
Notation as above. Then $\mathrm{ker} (\textsl{res}) = \C \Mod^{\tor}$.
\end{lemma}

\begin{proof}
Clearly, $\mathrm{ker} (\textsl{res})$ consists of $V \in \C \Mod$ such that $V_x = 0$ for all objects $x$ in $\D$. But every such $\C$-module $V$ is torsion. Indeed, for an arbitrary object $y$ in $\C$ and any $v \in V_y$, by the combinatorial conditions imposed on $\C$, we can find a maximal object $x$ in $\C$ and a morphism $\alpha: y \to x$ such that $\alpha \cdot v \in V_x$. But $x$ is also an object in $\D$, so $V_x = 0$, and consequently $v$ is a torsion element in $V$. This shows that $V$ is torsion, and hence $\mathrm{ker}(\textsl{res}) \subseteq \C \Mod^{\tor}$.

Conversely, if $V$ is a torsion $\C$-module, then $V_x = 0$ for all objects $x$ in $\D$. Otherwise, assume that there exists an object $x$ in $\D$ with $V_x \neq 0$. But since $V$ is torsion, for an arbitrary $0 \neq v \in V_x$, we can find a morphism $\alpha: x \to y$ such that $\alpha \cdot v = 0$. We must have $x= y$ and $\alpha$ is an isomorphism as $x$ is maximal. Consequently, $\alpha \cdot v = 0$ implies that $v = 0$, contradicting the assumption that $v \neq 0$. This shows that $\C \Mod^{\tor} \subseteq \mathrm{ker} (\textsl{res})$.
\end{proof}

\begin{remark} \normalfont
The previous equivalence may be obtained as a special case of the Comparison Lemma, see \cite[C2.2.3]{Jo}. A statement of the same fashion, for all finite EI categories with dense topology, may be found in \cite{WX}.
\end{remark}

\begin{example} \label{FS}
Let $\C$ be the following category: objects are $[n] = \{1, 2, \ldots, n\}$, $n \geqslant 1$, and morphisms are surjections. It is easy to see that $\C$ is a skeletal category of the category of finite sets and surjections, and it has a terminal object $[1]$. Moreover, $\C$ satisfies the left Ore condition. Therefore, by the previous theorem, we deduce that $\Sh(\BC^{\op}, R) \cong R \Mod$.
\end{example}

\subsection{The Nakayama functor} \label{type ii cats}

Recall that a skeletally small category $\C$ is called an \textit{EI category} if every endomorphism in $\C$ is an isomorphism. Examples of EI categories include free categories of quivers without oriented cycles, posets, groups, and some important categories constructed from groups such as transporter categories, orbit categories, Frobenius categories; see for instance \cite{Luck, Webb}.

In this subsection we assume that $\C$ is a skeletally small category satisfying the following conditions (called \textit{Type II combinatorial categories}):
\begin{itemize}
\item $\C$ is a small skeletal EI category;
\item every morphism in $\C$ is a monomorphism;
\item $\C$ is \textit{locally finite}, that is, $|\C(x, y)| < \infty$ for all objects $x$ and $y$;
\item $\C$ is \textit{inwards finite}, that is, for every object $x$, the set $\{y \in \C \mid \C(y, x) \neq \emptyset \}$ is finite.
\end{itemize}

Now let $R$ be a field. We also impose the following conditions on $\C$-modules:
\begin{itemize}
\item (LN): $\C$ is \textit{locally Noetherian} over $R$, that is, every finitely generated $\C$-module is Noetherian;
\item (LS): $\C$ is \textit{locally self-injective}, that is, every projective $\C$-module is also injective;\footnote{This notion is introduced in \cite{GLX}. We remind the reader that we do not require injective $\C$-modules to be projective.}
\item (NV): for every finitely generated $\C$-module $V$ which is not torsion, there exists a representable $\C$-module $P(x) = R\C(x, -)$ such that $\Hom_{\C \Mod} (V, P(x)) \neq 0$.
\end{itemize}

\begin{remark} \normalfont
At a first glance, the conditions imposed on the category $\C$ and $\C$-modules seem very strong and artificial. However, it is very natural for the practice purpose. Indeed, skeletal categories of $\FI$, $\OI$, and $\VI$ and their modules over a field of characteristic 0 all satisfy these requirements; see \cite{GLX} and \cite{GL1}. We also mention that the third condition (NV) imposed on $\C$-modules is slightly different from \cite[Statement (2) of Proposition 3.10]{GLX}. This is because for skeletal categories of $\FI$ and $\VI$, finitely generated torsion modules are precisely finite dimensional modules, but for skeletal categories of $\OI$, there exist finitely generated torsion modules with infinite dimension.
\end{remark}

For a Type II combinatorial category $\C$, We define the \textit{Nakayama functor} as follows:
\[
\nu: \C \module \to \C \fdmod, \quad V \mapsto D (\Hom_{\C \module} (V, \bigoplus_{x \in \Ob(\C)} P(x)))
\]
where $D= \Hom_{R} (-, R)$ is the classical duality functor, $\C \module$ and $\C \fdmod$ are the categories of finitely generated $\C$-modules and finite dimensional $\C$-modules respectively. We also define the \textit{inverse Nakayama functor} to be
\[
\nu^{-1}: \C \fdmod \to \C \module, \quad W \mapsto \Hom_{\C^{\op} \module} (DW, R\C(-, x)).
\]

The following facts have been established in \cite{GLX}: $\nu$ and $\nu^{-1}$ are indeed functors (see \cite[Subsection 2.2]{GLX}); the inwards finite condition guarantees that $\nu V$ is finite dimensional (see \cite[Lemma 2.1]{GLX}); the locally Noetherian condition (LN) guarantees that $\nu^{-1} W$ is finitely generated (see \cite[Lemma 2.2]{GLX}); the locally self-injective condition (LS) implies that $\nu$ is exact (see \cite[Lemma 2.11]{GLX}). Furthermore, $(\nu, \nu^{-1})$ is an adjoint pair (see \cite[Lemma 2.4]{GLX}), and $\nu \circ \nu^{-1}$ is isomorphic to the identity functor on $\C \fdmod$ (see \cite[Proposition 3.3]{GLX}). Besides, for every $x \in \Ob(\C)$, $P(x)$ is torsion free since every morphism is a monomorphism, so $\ker (\nu)$ contains all finitely generated torsion modules. On the other hand, by the third condition (NV) on $\C$-modules, $\ker (\nu)$ only contains torsion modules. Therefore, we have the following result:

\begin{theorem} \cite[Theorem 3.6]{GLX}
Suppose that $\C$ is a Type II combinatorial category, $\C \Mod$ satisfies conditions (LN), (LS), and (LV), and let $R$ be a field. Then the Nakayama functor induces the following equivalence
\[
\xymatrix{
\C \module / \C \module^{\tor} \ar[r]^-{\simeq} & \C \fdmod
}
\]
where $\C \module^{\tor}$ is the full subcategory of $\C \module$ of finitely generated torsion modules.
\end{theorem}

More explicitly, we have the following commutative diagram where $\bar{\nu}$ and $\bar{\nu}^{-1}$ give the equivalence $\C \module / \C \module^{\tor} \simeq \C \fdmod$, and $\nu \circ \textsl{inc}$ and $\nu^{-1}$ give the equivalence between the category $\C \module^{\sat}$ of finitely generated saturated $\C$-modules and $\C \fdmod$:
\[
\xymatrix{
& & \C \module^{\tor} \ar[d]^{\textsl{inc}}\\
\C \module^{\sat} \ar[rr]^{\textsl{inc}} & & \C \module \ar[rr]<.5ex>^{\nu} \ar[dd]<.5ex>^{\textsl{loc}} & & \C \fdmod \ar[ll]<.5ex>^{\nu^{-1}} \ar@{-->}[lldd]<.5ex>^{\bar{\nu}^{-1}}\\
 \\
& & \C \module / \C \module^{\tor} \ar[uu]<.5ex>^{\textsl{sec}} \ar@{-->}[rruu]<.5ex>^{\bar{\nu}}
}
\]
For more details, please refer to \cite{GLX}.

\begin{remark}
When $R$ is a field of characteristic 0 and $\C$ is a skeleton of $\FI$, the above equivalence was established by Sam and Snowden in \cite[Theorem 3.2.1]{SS} in the language of twisted commutative algebras. Nagpal asked in \cite{Nag1} whether the same conclusion holds for a skeleton of $\VI$, and this was answered affirmatively in \cite{GLX}. The above equivalence also holds for a skeleton of $\OI$ (see \cite[Theorem 1.2]{GL1}) and a skeleton of $\FI^m$, the product category of $\FI$ (see \cite[Theorem 1.2]{Zen}).
\end{remark}

\begin{remark}
If a Type II combinatorial category $\C$ is subcanonical with respect to the atomic topology, then each representable functor $P(x)=R\C(x,-)$ gives a sheaf of modules on $\C^{\op}$. Thus $\{P(x)\}_{x\in\Ob\C}$ is a set of generators of $\Sh(\BC^{\op},R)$. According to \cite[Page 92]{Po}, it makes sense to talk about ``$\{P(x)\}_x$-finitely generated'' objects in $\Sh(\BC^{\op},R)$. For obvious reasons, we shall simply call them ``\textit{finitely generated sheaves}''. Clearly, a sheaf is finitely generated if and only if it is a presheaf of finite type, see Proposition \ref{finite type}.
\end{remark}

Recall that $\mathrm{sh}(\BC^{\op}, R)$ is the (abelian) full subcategory of $\Sh(\BC^{\op},R)$ consisting of finitely generated sheaves.

\begin{corollary} \label{finitely generated sheaves}
Suppose that $\C$ is a Type II combinatorial category, $\C \Mod$ satisfies conditions (LN), (LS), and (LV), and let $R$ be a field. Then:
\begin{enumerate}
\item the following functors are equivalences of categories:
\[
\xymatrix{
\mathrm{sh}(\BC^{\op}, R) \ar[r]^-{=} & \C \module^{\sat} \ar[r]^-{\textsl{loc} \circ \textsl{inc}} & \C \module / \C \module^{\tor} \ar[r]^-{\bar{\nu}} & \C \fdmod;
}
\]

\item every object in $\mathrm{sh}(\BC^{\op}, R)$ is of finite length, and has finite injective dimension;

\item every indecomposable injective object in $\mathrm{sh}(\BC^{\op}, R)$ is isomorphic to an indecomposable direct summand of $P(x)$ for a certain $x \in \Ob(\C)$;

\item the functor $\nu$ induces a bijection between the set of isomorphism classes of simple objects in $\C \module^{\sat}$ and the set of isomorphism classes of simple $R\C(x, x)$-modules, where $x$ ranges over all objects in $\C$.
\end{enumerate}
\end{corollary}

Note that $\nu^{-1} \circ \nu$ and $\textsl{sec} \circ \textsl{loc}$ are functors from $\C \module$ to itself. However, since the images of these functors always lie in $\C \module^{\sat}$, by abuse of notations, we also regards them as functors from $\C \module$ to $\C \module^{\sat}$.

\begin{proof}
Statement (1) is follows from the finitely generated version of Theorem \ref{equivalences} restricted to $\C \module$ (which is abelian by the local Noetherianity of $\C$ over $R$) as well as the previous theorem.

Statement (2) is also clear since objects in $\C \fdmod$ satisfy these properties.

For Statement (3), we observe the following facts: $P(x)$ is a torsion free injective $\C$-module, and hence is contained in $\C \module^{\sat}$; $\nu P(x)$ is injective in $\C \fdmod$; and every indecomposable injective object in $\C \fdmod$ is isomorphic to a direct summand of $\nu P(x)$ for a certain object $x$ in $\C$. Now let $I$ be an indecomposable object in $\C \module^{\sat}$ (identified with $\mathrm{sh}(\BC^{\op}, R)$), then $\nu I$ is isomorphic to a direct summand of $\nu P(x)$ for a certain object $x$. Now applying the functor $\nu^{-1}$ we deduce that $I \cong (\nu^{-1} \circ \nu) I$ is isomorphic to a direct summand of $(\nu^{-1} \circ \nu)P(x) \cong P(x)$.

For Statement (4), we note that every simple $R\C(x, x)$-module can be viewed as a simple $\C$-module, and every simple $\C$-module is of this form. For details, please refer to \cite[Corollary (4.2)]{Webb} or \cite{Luck}.
\end{proof}

\begin{remark} \normalfont
The first conclusion in this corollary can be used to construct simple objects in $\mathrm{sh}(\BC^{\op}, R)$. Indeed, since for many examples it is easy to classify all simple objects in $\C \fdmod$, by applying the inverse Nakayama functor one obtains all simple objects in $\mathrm{sh}(\BC^{\op}, R)$.
\end{remark}

Now we give a few examples such that Corollary \ref{finitely generated sheaves} can apply.

\begin{example}
When $R$ is a field of characteristic 0, the skeletal subcategory of the category $\FI$ (resp., $\VI$) consisting of objects $[n]$ (resp., $\mathbb{F}_q^n$), $n \in \mathbb{N}$, satisfies all conditions specified in Corollary \ref{finitely generated sheaves}; see \cite{GLX}. More generally, the skeletal subcategory of $\FI^m$ (the product category of $m$ copies of $\FI$) consisting of objects $[n_1] \times \ldots \times [n_m]$, $n_i \in \mathbb{N}$, satisfies these conditions; see \cite{Zen}. When $R$ is a field, the skeletal subcategory of the category $\OI$ consisting of objects $[n]$, $n \in \mathbb{N}$, also satisfies these requirements; see \cite{GL1}.
\end{example}

A few other interesting examples of Type II combinatorial categories are introduced by Pol and Strickland in \cite{Pol}. Let $\mathcal{G}$ be the category of finite groups and conjugacy classes of surjective group homomorphisms. More explicitly, for two surjective homomorphism $\alpha: G \to H$ and $\beta: G \to H$, we say that $\alpha$ and $\beta$ lie in the same conjugacy class if there exists an element $h \in H$ such that $\alpha (g) = h\beta(g)h^{-1}$ for $g \in G$. A full subcategory $\mathcal{U}$ of $\mathcal{G}$ is called a \textit{multiplicative global family} if it is closed under taking finite products, quotient groups, and subgroups.

\begin{example}
Let $p$ be a prime number and $n$ a natural number. Then each of the following full subcategories of $\mathcal{G}$ is a multiplicative global family:
\begin{itemize}
\item $\mathcal{Z}$: the category of finite abelian groups;
\item $\mathcal{Z}[p^{\infty}]$: the category of finite abelian $p$-groups;
\item $\mathcal{G}[p^n]$: the category of finite groups of exponent dividing $p^n$;
\item $\mathcal{Z}[p^n]$: the category of finite abelian groups of exponent dividing $p^n$;
\item $\mathcal{Z}[p]$: the category of elementary abelian $p$-groups.
\end{itemize}
\end{example}

It is easy to see that the opposite categories of the categories listed in the above example are Type II combinatorial categories. Furthermore, Pol and Strickland proved the following result:

\begin{proposition} \cite[Propositions 12.16 and 15.1]{Pol}
Let $\mathcal{U}$ be a multiplicative global family, and let $R$ be a field of characteristic 0. Then torsion free injective objects in $\mathcal{U}^{\op} \Mod$ coincide with projective objects in it, and every torsion free object can be embedded into a projective one.
\end{proposition}

From this result we may deduce:

\begin{corollary}
Let $\mathcal{U}$ be a multiplicative global family, and let $R$ be a field of characteristic 0. Then the conditions (LS) and (NV) hold. Moreover, $\mathcal{U}$ satisfies the right Ore condition, so one can impose the atomic topology on it.
\end{corollary}

\begin{proof}
The first statement immediately follows from the previous proposition. Now we check that $\mathcal{U}$ satisfies the right Ore condition. Let $[f]: G \to H$ and $[f']: K \to H$ be two morphisms, which are represented by two surjective group homomorphisms $f$ and $f'$ respectively. Then the product $G \times H$ is contained in $\mathcal{U}$ since it is closed under products, so is the fiber product $G \times_K H$ as $\mathcal{U}$ is closed under taking subgroups. One can check that $G \times_K H$ is also the pull-back of $[f]$ and $[f']$ in $\mathcal{U}$. Thus $\mathcal{U}$ has pull-backs, and hence satisfies the right Ore condition.
\end{proof}

Unfortunately, it is still unknown whether $\mathcal{U}^{\op}$ has the locally Noetherian property (LN) over a field $R$ of characteristic 0. Pol and Strickland show that if $\C$ is the opposite category of $\mathcal{Z}[p^{\infty}]$ or $\mathcal{Z}[p^n]$, then $\C$ has this property; see \cite[Theorem 13.15]{Pol}.

In a summary, one has:

\begin{proposition} \label{equivalence of Z}
Suppose that $R$ is a field of characteristic 0, and let $\C$ be a skeletal category of the following categories:
\[
\FI, \, \VI, \, \OI, \, \FI^m, \, \mathcal{Z}[p^{\infty}]^{\op}, \, \mathcal{Z}[p^n]^{\op}.
\]
Then one has
\[
\mathrm{sh}(\BC^{\op}, R) \simeq \C \fdmod
\]
\end{proposition}

\begin{proof}
Apply Corollary \ref{finitely generated sheaves}.
\end{proof}

\begin{remark} \normalfont
The requirement that the field $R$ has characteristic 0 can be dropped for $\OI$. For $\mathcal{Z}[p^{\infty}]^{\op}$, $\mathcal{Z}[p^n]^{\op}$, one only requires that the characteristic of $R$ is distinct from $p$.
\end{remark}

\subsection{Local cohomology}

Theorem \ref{sheaf cohomology} relates sheaf cohomology functors to derived functors of the torsion functor $\tau$. In this subsection we give another interpretation of sheaf cohomology for some special categories in terms of local cohomology. For this purpose, we introduce the following definitions.

Let $\C$ be a small skeletal EI category. The \textit{category algebra} $R\C$ is a free $R$-module with basis elements all morphisms in $\C$, and multiplication is given by the following rule: for two morphisms $\alpha$ and $\beta$, we define $\alpha \cdot \beta = \alpha \circ \beta$ if they can be composed, and 0 otherwise. Note that $R\C$ is an associative algebra, and it has a multiplicative unit if and only if $\C$ has only finitely many objects.

Now impose the following conditions on $\C$ (called \textit{Type III combinatorial categories}):
\begin{itemize}
\item $\C$ is a small skeletal EI category;
\item objects in $\C$ are parameterized by $\N$ (that is, $\Ob(\C) = \{x_n \mid n \in \N \}$);
\item $\C(x_m, x_n) \neq \emptyset$ if and only if $m \leqslant n$;
\item the group $G_n = \C(x_n, x_n)$ acts transitively on $\C(x_m, x_n)$ form the left hand side.
\end{itemize}

Examples of Type III combinatorial categories include skeletal subcategories of $\FI$ and $\VI$, see \cite{GL1, GLX}.

Now let let $\C$ be a Type III combinatorial category, and let $\mathfrak{m}$ be the free $R$-module spanned by all non-invertible morphisms in $\C$. It is easy to check that $\mathfrak{m}$ is actually a two-sided ideal of $R\C$. We use this two-sided ideal to define the local cohomology of $\C$-modules, as the second author and Ramos did for the category $\FI$ in \cite{LR}.

Note that there is a natural fully faithful functor from $\C \Mod$ to $R\C \Mod$, so we may regard a $\C$-module as an $R\C$-module. But the converse is not true. Instead, an $R\C$-module $V$ is a $\C$-module if and only if one has the following decomposition:
\[
V = \bigoplus_{x \in \Ob(\C)} 1_xV
\]
For details, please refer to \cite{Luck, Mit, Webb}. Since $R\C$ as a free $R$-module is spanned by all morphisms in $\C$, every element in $R\C$ is a finite $R$-linear combination of morphisms. Consequently, it is the direct sum of free $R$-modules $R\C(-, x)$ spanned by all morphisms ending at $x$ with $x$ ranging over all objects. But $R\C(-, x)$ is precisely $1_x R\C$, so we have
\[
R\C = \bigoplus_{x \in \Ob(\C)} 1_xR\C,
\]
and hence $R\C$ is a $\C$-module. Similarly, for every positive $n$, $\mathfrak{m}^n$ is also a $\C$-module as it is spanned by morphisms which can be written as a composite of $n$ non-invertible morphisms. Consequently, the quotient $R\C / \mathfrak{m}^n$ is a $\C$-module, and has the following decomposition
\[
R\C / \mathfrak{m}^n = \bigoplus_{x \in \Ob(\C)} P(x)/\mathfrak{m}^n P(x)
\]
as a $\C$-module where $P(x) = R\C(x, -)$.

For $n \geqslant 1$, we define the functor
\[
\fHom_{\C \Mod} (R\C/\mathfrak{m}^n, \bullet): \C \Mod \to \C \Mod, \quad V \mapsto \bigoplus_{x \in \Ob(\C)} \Hom_{\C \Mod} (P(x)/\mathfrak{m}^n P(x), V).
\]
Note that this is well defined since $\bigoplus_{x \in \C} (P(x)/\mathfrak{m}^n P(x), V)$ carries the structure of $\C$-module, and in general
\[
\Hom_{\C \Mod}(R\C / \mathfrak{m}^n, V) \ncong \fHom_{\C \Mod} (R\C/\mathfrak{m}^n, V)
\]
as the former one might not be a $\C$-module (although it is an $R\C$-module). For details, please refer to \cite[Definition 5.4 and Remark 5.5]{LR}, where the construction still holds in this more general setting.

The functor $\fHom_{\C \Mod} (R\C/\mathfrak{m}^n, \bullet)$ is left exact, and hence has right derived functors
\[
\fExt_{\C \Mod}^i (R\C/\mathfrak{m}^n, \bullet) = \bigoplus_{x \in \C} \Ext_{\C \Mod}^i (P(x)/\mathfrak{m}^n P(x), \bullet).
\]
Finally, we define the local cohomology functors as
\[
\HH^i_{\mathfrak{m}} (\bullet) = \varinjlim_n \fExt_{\C \Mod}^i (R\C/\mathfrak{m}^n, \bullet).
\]

\begin{theorem} \label{local cohomology}
Suppose that $\C$ is a Type III combinatorial category, and let $V \in \Sh(\BC^{\op}, R)$. Then for every object $x$ in $\C$ and $i \geqslant 1$, one has
\[
\mathrm{R}^i \Gamma_x(V) \cong (\HH_{\mathfrak{m}}^{i+1}(V))_x,
\]
the value of the $\C$-module $\HH_{\mathfrak{m}}^{i+1}(V)$ on the object $x$.
\end{theorem}

\begin{proof}
We show that
\[
\mathrm{R}^i \tau (\bullet) \cong \HH_{\mathfrak{m}}^i(\bullet)
\]
for $i \geqslant 1$ as functors on $\C \Mod$, where $\tau$ is the torsion functor. Then the conclusion follows from Theorem \ref{sheaf cohomology}. Since $\mathrm{R}^i \tau (\bullet)$ and $\HH_{\mathfrak{m}}^i (\bullet)$ are the $i$-th derived functors of $\tau$ and $\HH_{\mathfrak{m}}^0(\bullet)$ respectively, we only need to show
\[
\HH_{\mathfrak{m}}^0 (\bullet) = \varinjlim_n \fHom_{\C \Mod} (R\C/\mathfrak{m}^n, \bullet) \cong \tau(\bullet).
\]
That is, for every $W \in \C \Mod$, $\HH_{\mathfrak{m}}^0 (W) = \tau W = W_T$, or equivalently, $(\HH_{\mathfrak{m}}^0(W))_x = (W_T)_x$ for every object $x$ in $\C$.

It is easy to see that $\Hom_{\C \Mod} (P(x)/\mathfrak{m}^n P(x), W)$ consists of elements $w$ (the images of the module homomorphism determined by $1_x \mapsto w$) in $W_x$ such that $f \cdot w = 0$ for all $f \in \mathfrak{m}^n$. Therefore,
\[
\Hom_{\C \Mod} (P(x)/\mathfrak{m}^n P(x), W) \subseteq (W_T)_x,\]
and hence
\[
(\HH_{\mathfrak{m}}^0(W))_x = \varinjlim_n \Hom_{\C \Mod} (P(x)/\mathfrak{m}^n P(x), W) = \bigcup_{n \geqslant 1} \Hom_{\C \Mod} (P(x)/\mathfrak{m}^n P(x), W) \subseteq (W_T)_x.
\]
On the other hand, if $w \in W_x$ is a torsion element, then there exists a morphism $f: x \to y$ such that $f \cdot w = 0$. Clearly, $y \neq x$ since otherwise $f$ is an isomorphism in $\C$. Thus $f \in \mathfrak{m}$. Since objects in $\C$ are parameterized by $\N$, we can find $m, n \in \N$ such that $x = x_m$ and $y = x_n$ with $m < n$. Let $s = n -m$. We claim that $\mathfrak{m}^s \cdot v = 0$. If this holds, then
\[
w \in \Hom_{\C \Mod} (P(x)/\mathfrak{m}^s P(x), W) \subseteq \varinjlim_n \Hom_{\C \Mod} (P(x)/\mathfrak{m}^n P(x), W) = (\HH_{\mathfrak{m}}^0(W))_x,
\]
and the conclusion of this theorem is established.

Now we prove the claim. Let $g: x_a \to x_b$ be a morphism contained in $\mathfrak{m}^s$. If $a \neq m$, then $g \cdot w = 0$ trivially holds, so we can assume that $a = m$. Since $g \in \mathfrak{m}^s$, it can be written as a composite of $s$ non-invertible morphisms. Consequently, $b \geqslant a + s = n$, so we can find a morphism $h: y = x_n \to x_b$. Note that $hf \in \C(x_m, x_b)$. Since $\C(x_b, x_b)$ acts transitively on $\C(x_m, x_b)$, we can also find an automorphism $\delta: x_b \to x_b$ satisfying $\delta hf = g$. Therefore, $g \cdot w = (\delta h f) \cdot w = 0$ as well. This finishes the proof of the claim.
\end{proof}

We have been reinterpreting various  sheaf-theoretic constructions by representation theory. We take the opportunity to state one example, which goes in the opposite direction.

\begin{remark}
Let $\C$ be the skeletal subcategory of $\FI$ consisting of objects $[n]$, $n \in \N$. Then by direct calculation we have an isomorphism of categories $[i]\backslash \C \cong \C$, $\forall i$, where $[i]\backslash\C$ is the category under $[i]$ and $([i]\backslash\C)^{\op}\cong\C^{\op}/[i]$. The localization morphism \cite[Definition 18.19.1]{Stack} at $[i]$, $j_{[i]}=(j_{[i]}^*,{j_{[i]}}_*)$, contains the following restriction along $\C^{\op}/{[i]}\to\C^{\op}$
$$
j^*_{[i]}: \C\Mod=\PSh(\C^{\op},R) \to \PSh(\C^{\op}/[i], R)=\C\Mod,
$$
which is precisely the (exact) \textit{shift functor} $\Sigma^i$ introduced in \cite{CEF}, satisfying $\Sigma P([n]) \cong P([n]) \oplus P([n-1])^n$. The shift functors $\Sigma^i$ play an important role in the representation theory of FI.
\end{remark}

\section{Discrete representations of topological groups} \label{discrete repns section}

In this section we extend the classical result of Artin from sheaves of sets to sheaves of modules, serving as a bridge connecting sheaves of modules over atomic sites and discrete representations of topological groups. Throughout this section let $G$ be a topological group, and let $X$ be a set equipped with the discrete topology. Note that $\Aut(X)$ is a topological group with a topology inherited from the product topology on $X^X$. Let $\rho: G \to \Aut(X)$ be a group homomorphism. We say that the action of $G$ on $X$ induced by $\rho$ is \textit{discrete} if $\rho$ is continuous with respect to the topologies of $G$ and $\Aut(X)$, and call $X$ a \textit{discrete $G$-set}.\footnote{In \cite[III.9]{MM} the action is called continuous, and $X$ is called a continuous $G$-set.} It is easy to see that $\rho$ is continuous if and if only for every $x \in X$, the stablizer subgroup $\Stab_G(x)$ is open.

\subsection{Equivalences between representation categories and sheaf categories} \label{Artin's theorem}

Following \cite[III.9]{MM}, let $\mathbf{B}G$ be the category of discrete $G$-sets, and morphisms between two $G$-sets are $G$-equivariant maps (which are automatically continuous since $G$-sets are equipped with the discrete topology). Let $\mathbf{S}G$ be the full subcategory of $\mathbf{B}G$ consisting of objects of the form $G/H$ where $H$ is an open subgroup (thus the quotient topology on $G/H$ is discrete). Morphisms in $\mathbf{S}G$ are induced by elements in $G$. Explicitly, every morphism from $G/H$ to $G/K$ is of the form $\alpha_g: xH \mapsto xg^{-1}K$ such that $gHg^{-1} \leqslant K$. For a cofinal system of open subgroups $\mathcal{U}$ (that is, every open subgroup contains a member in $\mathcal{U}$), let $\mathbf{S}_{\mathcal{U}}G$ be the full subcategory of $\mathbf{S}G$ consisting of objects of the form $G/H$ with $H \in \mathcal{U}$.  In the literature $\mathbf{S}G$ and $\mathbf{S}_{\mathcal{U}}G$ are called \textit{orbit categories} of $G$ with respect to open subgroups and $\mathcal{U}$ respectively.

\begin{theorem} \cite[Theorems 1 and 2]{MM} \label{mm}
Let $G$ be a topological group and $\mathcal{U}$ be a cofinal system of open subgroups of $G$. Then one has the following equivalences of categories
\[
\Sh(\mathbf{S}G) \simeq \mathbf{B}G \simeq \Sh(\mathbf{S}_{\mathcal{U}}G),
\]
where all sheaves are taken with respect to the atomic topology.
\end{theorem}

The first equivalence is given by two functors
\begin{align*}
\phi: \mathbf{B}G \to \Sh(\mathbf{S}G),\\
\psi: \Sh(\mathbf{S}G) \to \mathbf{B}G.
\end{align*}
Explicitly, given $X \in \mathbf{B}G$ and an open subgroup $H$, one lets
\[
\phi(X)(H) = X^H = \{ x \in X \mid h \cdot x = x, \, \forall h \in H \}.
\]
Conversely, given $F \in \Sh(\mathbf{S}G)$, one sets
\[
\psi(F) = \varinjlim_{H} F(H)
\]
where the colimit is taken over the set of all open subgroups ordered by inclusion. The second equivalence can be obtained similarly.

Now let $R$ be a commutative ring. Since a sheaf of $R$-modules is considered as a contravariant functor into $R \Mod$ meeting the existence and uniqueness condition of amalgamations, the preceding equivalences in Theorem \ref{mm} automatically induce the following equivalences:
\[
RG \DMod \simeq \Sh(\mathbf{S}G, R) \simeq \Sh(\mathbf{S}_{\mathcal{U}}G, R)
\]
where all sheaves are taken over the atomic topology, and $RG \DMod$ is the category of all \textit{discrete representations} $V$ of $G$ (or discrete $RG$-modules); that is, $V$ is an $RG$-module as well as a discrete $G$-set, or equivalently, for every $v \in V$, the stabilizer subgroup $\Stab_G(v)$ is open. By this criterion, submodules and quotient modules of discrete modules are also discrete. Furthermore, the $R$-linearization of a discrete $G$-set is a discrete representation.

\begin{lemma} \label{discrete sets and modules}
Let $X$ be a discrete $G$-set and let $R$ be a commutative ring. Then the free $R$-module $\underline{X}$ (with discrete topology) spanned by $X$ is a discrete representation of $G$.
\end{lemma}

\begin{proof}
The induced action is defined as follows: for $v = r_1 x_1 + \ldots + r_n x_n \in \underline{X}$ with $r_i \in R$ and $x_i \in X$, and $g \in G$, one sets $g \cdot v = r_1 (g \cdot x_1) + \ldots r_n (g \cdot x_n)$. Note that $\underline{X}$ is a discrete representation if and only if $\Stab_G(v)$ is an open subgroup of $G$. However,
\[
\Stab_G(x_i) \cap \ldots \cap \Stab_G(x_n) \leqslant \Stab_G(v),
\]
and each $\Stab_G(x_i)$ is open. Therefore, $\Stab_G(v)$ is also open since it contains an open subgroup.
\end{proof}

\begin{remark} \normalfont
It is easy to see that $\Stab_G(v)$ acts on $\{x_i \mid i \in [n] \}$, and $\Stab_G(x_i) \cap \ldots \cap \Stab_G(x_n)$ is a normal subgroup of $\Stab_G(v)$. Moreover, the quotient group acts on this finite set faithfully, so it is isomorphic to a subgroup of $S_n$. Consequently, $\Stab_G(x)$ is a finite disjoint union of open sets, each of which is homeomorphic to $\Stab_G(x_i) \cap \ldots \cap \Stab_G(x_n)$.
\end{remark}

\begin{remark} \normalfont
The regular $RG$-module $RG$ is discrete if and only if $G$ is equipped with discrete topology. Indeed, since $\Stab_G(1_R1_G) = 1_G$, $RG$ is a discrete representation only if the trivial subgroup is open, if and only if $G$ is equipped with the discrete topology.
\end{remark}

For the convenience of the reader, let us restate several results, whose application will be illustrated by a few examples described in later subsections.

\begin{theorem} \label{combination}
Let $G$ be a topological group, $\mathcal{U}$ be a cofinal system of open subgroups of $G$, $R$ a commutative ring, and $\C = (\mathbf{S}_{\mathcal{U}}G)^{\op}$. Then one has:
\[
RG \DMod \simeq \Sh(\BC^{\op}, R) \simeq \C \Mod / \C \Mod^{\tor}.
\]
Furthermore, if $\C^{\op}$ is a Type II combinatorial category, then one has:
\[
\mathrm{sh}(\BC^{\op}, R) \simeq \C \module / \C \module^{\tor} \simeq \C \fdmod.
\]
\end{theorem}

\subsection{Infinite symmetric group}

Representation theory of the infinite symmetric group has attracted the attention and has been extensively studied by a lot of people; see for instance \cite{BEH, Lieb, Nag2, Ol, Ok, SS, SS2}. In particular,  Liebermann proved in \cite{Lieb} that the category of unitary representations of $\Aut(\N$) is semisimple, and classified all simple unitary representations; Sam and Snowden constructed all indecomposable injective and simple discrete representations over the complex field in \cite{SS, SS2}; Nagpal even constructed all simple discrete representations over any field in \cite{Nag2}. Therefore, we do not claim originality for these results, and correspondingly, we omit their proofs. Indeed, the main goal of this subsection is not to provide new classifications. Instead, we use it as an example to illustrate our strategy (built upon Theorem \ref{combination}) of studying discrete representations of topological groups through sheaf theory and representation theory of categories.

\subsubsection{Orbit categories}
Let $G = \Aut(\N)$ be the set of all permutations on $\N$ with the topology inherited from the product topology of $\N^{\N}$, and let $\mathcal{U}$ be the set of pointwise stabilizer subgroups $U_S = \mathrm{Stab}_G(S)$ where $S$ is a finite subset of $\N$. Then $G$ is a topological group and $\mathcal{U}$ is a cofinal system of open subgroups of $G$. Furthermore, $\mathbf{S}_{\mathcal{U}}G$ is equivalent to the full subcategory of $\FI^{\op}$ consisting of finite subsets of $\N$, which is equivalent to $\FI^{\op}$. Consequently, Artin's theorem tells us
\[
\mathbf{B}\Aut(\N) \simeq \Sh(\FI^{\op})
\]
which is called the \textit{Schanule topos}. For details, see \cite[III.9]{MM} as well as \cite[Examples A2.1.11(h) and D3.4.10]{Jo}.

The group $\Aut(\N)$ has a dense subgroup $H = S_{\infty}$ consisting of permutations on $\N$ fixing all but finitely many elements in $\N$; that is, $H = \varinjlim_n S_n$, called the \textit{infinite symmetric group} over countably infinitely many elements. Equip $H$ with the topology inherited from $\Aut(\N)$. Then $\mathcal{U} = \{ \Stab_H(S) \mid S \subset \N, \, |S| < \infty\}$ is a fundamental system of open subgroups of the identity in $H$. As in \cite[III.9]{MM}, one can show that $\mathbf{S}_{\mathcal{U}}H$ is also equivalent to $\FI^{\op}$. Consequently, the categories of discrete representations of these two groups are equivalent by Artin's theorem.

\subsubsection{Simple and indecomposable objects}

In the rest of this subsection let $G = S_{\infty}$ and $\C$ be the skeletal category of $\FI$ consisting of objects $[n]$, $n \in \N$. Then we have:
\[
RG \DMod \simeq \Sh(\BC^{\op}, R) \simeq \C \Mod /\C \Mod^{\tor}.
\]
Since $\C^{\op}$ is subcanonical with respect to the atomic topology; that is, each representable functor $R\C([n], -)$ is contained in $\Sh(\BC^{\op}, R)$, it corresponds to a discrete representation $F(n)$ of $G$, which can be explicitly described as follows: as a free $R$-module it is spanned by the set of all injective maps from $[n]$ to $\N$, and the natural action of $G$ on this set uniquely determines the action of $G$ on $F(n)$; see Lemma \ref{discrete sets and modules}. Alternatively, $F(n)$ can be described as an induced module:
\[
F(n) \cong RG \otimes_{RS_{\infty}'} R,
\]
where $S_{\infty}'$ is the subgroup of $S_{\infty}$ consisting of elements fixing all members in $[n]$.

We say that a discrete $RG$-module $V$ is \textit{finitely generated} if $V$ is isomorphic to a quotient of $\bigoplus_{n \in \N} (F(n))^{c_n}$ with $\sum_{n \in \N} c_n < \infty$. Note that the kernel of the quotient map is also discrete and finitely generated by the equivalence $RG \DMod \cong \C \Mod^{\sat}$ and the locally Noetherian property of $R\C$. Thus the full subcategory of finitely generated discrete $RG$-modules is an abelian subcategory of $RG \DMod$, denoted by $RG \dmodule$. We shall warn the reader that here the meaning of finitely generated is different from its usual sense (the algebraic sense) in representation theory: for instance, the regular $RG$-module is finitely generated in the algebraic sense, but is not a discrete module, and hence is not a finitely generated discrete module

With this definition, when $R$ is a field of characteristic 0, we have the following equivalences:
\[
RG \dmodule \simeq \mathrm{sh}(\BC^{\op}, R) \simeq \C \module /\C \module^{\tor} \simeq \C \fdmod
\]
which were established by Sam and Snowden in \cite[Theorem 2.5.1]{SS} and \cite[Corollary 6.2.5]{SS2} (except the equivalence involving the sheaf category). Simple objects and indecomposable injective objects in $\Sh(\BC^{\op}, R)$ have essentially been classified in \cite[Corollaries 2.2.7 and 2.2.15]{SS} as well. However, since these results are described in the language of twisted commutative algebras, here we give a slightly different version in the language of representations of $\C$, which might be easier to understand for some readers.

By Statement (4) of Proposition \ref{finitely generated sheaves}, an indecomposable injective object in $\Sh(\BC^{\op}, R)$ is precisely an indecomposable projective $\C$-module, which have been classified in \cite[Subsection 2.2]{CEF}. Explicitly, for any $m \in \N$ and a partition $\lambda: \lambda_1 \geqslant \lambda_2 \geqslant \ldots \geqslant \lambda_r > 0$ of $[m]$, let $L_{\lambda}$ be the simple $RS_m$-module corresponding to $\lambda$. Then one can construct an indecomposable projective $\C$-module
\[
P(\lambda) \cong R\C \otimes_{RS_m} L_{\lambda} \cong P(m) \otimes_{RS_m} L_{\lambda}
\]
where $RS_m$ is viewed as a subalgebra of $R\C$ and $P(m) = R\C([m], -)$. Furthermore, $(P(\lambda))_n = 0$ for $n < m$, and for $n \geqslant m$,
\[
(P(\lambda))_n = \bigoplus_{\mu} L_{\mu}
\]
where $\mu$ is the partition of $[n]$ whose corresponding Young diagram $Y_{\mu}$ is obtained by adding one box to $n - m$ distinct columns of $Y_{\lambda}$.

The partition $\lambda$ induces a \textit{uniform partition} $\boldsymbol{\lambda}$ for all $n \geqslant m + \lambda_1$. Explicitly, $\boldsymbol{\lambda}(n)$ is the following partition
\[
n - m \geqslant \lambda_1 \geqslant \lambda_2 \geqslant \ldots \geqslant \lambda_r,
\]
of $[n]$. Denote the simple $RS_n$-module corresponding to this partition by $L_{\boldsymbol{\lambda}(n)}$. When $n < m + \lambda_1$, we let $L_{\boldsymbol{\lambda}(n)}$ be the zero module by convention. In \cite[Proposition 3.4.1]{CEF}, Church, Ellenberg, and Farb described a $\C$-module $L_{\boldsymbol{\lambda}}$ of the following structure:
\[
\xymatrix{
L_{\boldsymbol{\lambda}(0)} \ar[r] & L_{\boldsymbol{\lambda}(1)} \ar[r] & L_{\boldsymbol{\lambda}(2)} \ar[r] & \ldots
}
\]
Then a translation of \cite[Corollary 2.2.7]{SS} asserts that these modules exhaust all simple objects in $\Sh(\BC^{\op}, R)$ up to isomorphism.

Simple discrete representations of $S_{\infty}$ have been classified in \cite[Proposition 6.1.5]{SS2}. We now construct indecomposable injective discrete representations. Again, they (up to isomorphism) are one to one corresponded to partitions $\lambda$ of $[n]$, where $n$ ranges over all natural numbers. Regard $L_{\lambda}$ as a simple representation of the subgroup $H = S_n \times S_{\infty}'$, where $S_{\infty}'$ acts trivially on $L_{\lambda}$. Define $I(\lambda) = RG \otimes_{RH} L_{\lambda}$. Then these modules exhaust all indecomposable injective discrete representation of $S_{\infty}$ up to isomorphism.

\begin{remark} \normalfont
We shall point out that these constructions have been long known and heavily studied in representation theory of infinite symmetric groups; see for instances \cite{Lieb, Ol}. In particular, simple objects $L_{\boldsymbol{\lambda}}$ were observed and studied by Ol'sanskii fifty year ago (see \cite[Theorem 2.7]{Ol}), though he used them to consider unitary representations of infinite symmetric groups rather than representations of $\FI$ (\cite[Proposition 3.4.1]{CEF}) or sheaves of modules over $\FI$.
\end{remark}

As a summary, we restate the following classification result:

\begin{proposition}
Let $R$ be a field of characteristic 0. Then up to isomorphism, one has:
\begin{enumerate}
\item simple objects in $\Sh(\BC^{\op}, R)$ are those $L_{\boldsymbol{\lambda}}$;
\item indecomposable injective objects in $\Sh(\BC^{\op}, R)$ are those $P(\lambda)$;
\item indecomposable injective discrete representations of $S_{\infty}$ or $\Aut(\N)$ are those $I(\lambda)$.
\end{enumerate}
\end{proposition}

\begin{remark} \normalfont
From this proposition we know that when $R$ is a field of characteristic 0 and $G$ is $\Aut(\N)$ or $S_{\infty}$, the Grothendieck group
\[
K_0(RG \DMod) \cong \bigoplus_{n \in \N} K_0(RS_n \Mod).
\]
Furthermore, by the work of Zelevinsky described in \cite{Ze}, $\mathscr{H} = \bigoplus_{n \in \N} K_0(RS_n \Mod)$ has a special Hopf algebra structure, called a \textit{positive self-adjoint Hopf algebra}, whose multiplication and comultiplication are induced by the induction and restriction of representations of symmetric groups respectively. By this isomorphism, one can also impose a positive self-adjoint Hopf algebra structure on $K_0(RG \DMod)$. For details about positive self-adjoint Hopf algebras, please refer to \cite{Ze}. The K-theory of this category was also studied in \cite[Subsection 2.4]{SS}, where a ring structure on $K_0(RG \DMod)$ is described in \cite[Remark 2.4.6]{SS} without relying on the above isomorphism. We wonder whether the Hopf structure on this Grothendieck group can also be defined in a direct and natural way.
\end{remark}

\subsubsection{Representation stability} \label{stability result}

A main incentive for people to study representation theory of some combinatorial categories such as $\FI$ is the representation stability phenomena, first observed in \cite{CEF}, which has rich applications in investigating asymptotic behavior of (co)homology groups of topological spaces and geometric groups; see for instances\cite{CE, CEF, CEFN, Gad, JW, MW, MPW}. Now we describe some stability properties of discrete representations of $G$, which is $\Aut(\N)$ or $S_{\infty}$.

Let $V$ be a discrete $RG$-module. We can construct a presentation
\[
\bigoplus_{T \subseteq \mathbb{N}} R(G/\mathrm{Stab}(T))^{c_T} \longrightarrow \bigoplus_{T \subseteq \mathbb{N}} R(G/\mathrm{Stab}(T))^{d_T} \longrightarrow V \longrightarrow 0
\]
where $T$ is a finite subset of $\mathbb{N}$ and $c_T$ and $d_T$ are the multiplicities. For each such presentation one obtains a number (which may be infinity)
\[
\sup \{ |T| \mid T \text{ appears in the above presentation} \},
\]
and one defines the \textit{presentation degree} $N_V$ to be the smallest one. We say that $V$ is \textit{presented in finite degrees} if $N_V$ is finite.

\begin{proposition} \label{stability of Aut(N)}
Let $V$ be a discrete representation of $G$ over an arbitrary commutative ring $R$ and suppose that the presentation degree $N_V$ of $V$ is finite. Then one has
\[
V = \bigcup_{\substack{T \subseteq \N \\ |T| \leqslant N_V }} V^{U_T}.
\]
Furthermore, $N_V$ is the minimal number such that the above identity holds.
\end{proposition}

\begin{proof}
Let $\phi V$ be the corresponded sheaf where $\phi: \mathbf{B}G \to \Sh(\mathbf{S}G)$ is the equivalence of categories given by Artin's theorem. Then by Theorem \ref{mm} we have
\[
V = \bigcup_{\substack{S \subseteq \N \\ |S| < \infty }} (\phi V)_S.
\]
Note that $R(G/\mathrm{Stab}(T))$ is sent to the projective $(\mathbf{S}G)^{\op}$-module $R\mathbf{S}G(-, T)$ by $\phi$, and vice versa by $\psi$ (taking direct limits). Thus the presentation degree (defined in \cite{GL}) of the $(\mathbf{S}G)^{\op}$-module $\phi V$ is also $N_V$. Consequently, by \cite[Theorem C]{CEFN} or \cite[Theorem 3.2]{GL}, for $N \geqslant N_V$ and any finite set $S$, one has
\[
(\phi V)_S = \bigcup_{\substack{T \subseteq S \\ |T| \leqslant N}} (\phi V)_T.
\]
Moreover, $N_V$ is the minimal number such that the above equality holds by \cite[Theorem 3.2]{GL}. Combining these two equalities, we obtain the desired equality.
\end{proof}

\begin{remark}
The minimal number $N$ such that the above equality holds is called the \textit{stable range} of $\phi V$ in \cite{CEFN}. Thus the above proposition asserts that the stable range equals to the presentation degree.
\end{remark}

\subsection{Infinite general or special linear groups}

In this subsection we consider discrete representations of infinite general or special linear groups over finite fields. Representation theory of these group have been extensively studied in the literature; see for instances \cite{Nag1, Nag2, PS, SS2}. Thus we do not claim originality for the classification results, and correspondingly we omit their proofs. Instead, we use these groups as another example to illustrate our approach.

\subsubsection{Orbit categories}

Let $\mathbb{F}_q$ be a finite field, and let $\mathbb{F}_q^{\N}$ be the vector space with basis elements indexed by $\N$; or equivalently,
\[
\mathbb{F}_q^{\N} = \{ f: \N \to \mathbb{F}_q \mid f(n) = 0 \text{ for } n \gg 0 \}.
\]
Equip $\mathbb{F}_q^{\N}$ with the discrete topology, and let $\Aut(\mathbb{F}_q^{\N})$ be the group consisting of invertible maps from $\mathbb{F}_q^{\N}$ to itself. As explained in the previous subsection, we can equip $\Aut(\mathbb{F}_q^{\N})$ with the pointwise convergence topology. In particular, the following open subgroups form a fundamental system of open neighborhoods of the identity element:
\[
\Stab_{\Aut(\mathbb{F}_q^{\N})}(S) = \{f \in \Aut(\mathbb{F}_q^{\N}) \mid f(s) = s, \, \forall s \in S \}
\]
where $S$ ranges over all finite subsets of $\mathbb{F}_q^{\N}$.

Note that $\Aut(\mathbb{F}_q^{\N})$ has the following subgroup respecting the vector space structure:
\[
G = \mathrm{GL}(\mathbb{F}_q^{\N}) = \{f \in \Aut(\mathbb{F}_q^{\N}) \mid f \text{ is a $\mathbb{F}_q$-linear} \}.
\]
A fundamental system of open neighborhoods of the identity is given by open subgroups
\[
\Stab_G(S) = \{f \in G \mid f(s) = s, \, \forall s \in S \} = \Stab_G (\underline{S}),
\]
where $\underline{S}$ is the vector space spanned by a finite subset $S$ of $\mathbb{F}_q^{\N}$.

Let $\mathcal{U} = \{\Stab_G(S) \mid S \subset \mathbb{F}_q^{\N}, \, |S| < \infty \}$, which is a cofinal system of open subgroups of $G$.

\begin{proposition}\label{VI}
Notation as above. Then $\mathbf{S}_{\mathcal{U}}G$ is equivalent to the category $\VI^{\op}$.
\end{proposition}

\begin{proof}
Define a contravariant functor $\rho: \mathbf{S}_{\mathcal{U}} G \to \VI$ as follows. It sends an object $G/\Stab_G(S)$ to $\underline{S}$ in $\VI$. Clearly, $\rho$ is essentially surjective. Let
\[
\alpha \in \mathbf{S}_{\mathcal{U}}G (G/\Stab_G(S), G/\Stab_G(T))
\]
be a morphism, which can be represented by an element $g \in G$ with $g \Stab_G(S) g^{-1} \subseteq \Stab_G(T)$. This happens if and only if $g^{-1} (T) \subseteq S$; that is, the restriction of $g^{-1} |_{\underline{T}}$ is contained in $\VI(\underline{T}, \underline{S})$. We thus set $\rho(\alpha) = g^{-1} |_{\underline{T}}$.

We prove that this functor is well defined, faithful, and full. Suppose that
\[
\alpha \in \mathbf{S}_{\mathcal{U}}G (G/\Stab_G(S), G/\Stab_G(T))
\]
are represented by two elements $g, h \in G$. This happens if and only if $g^{-1} \Stab_G(T) = h^{-1} \Stab_G(T)$, and if and only if $hg^{-1} \in \Stab_G(T)$. But this is true if and only if $h^{-1}(t) = g^{-1}(t)$ for all $t \in T$; that is, $h^{-1} |_{\underline{T}} = g^{-1} |_{\underline{T}}$ as $g$ and $h$ are linear maps. Thus $\rho$ is well defined. This argument also shows that $\rho$ is faithful. Since every morphism in  $\VI(\underline{T}, \underline{S})$ can be extended to an element in $G$, the functor $\rho$ is also full. The conclusion then follows.
\end{proof}

The group $\mathrm{GL}(V)$ has a dense subgroup $\mathrm{GL}_{\infty} (\mathbb{F}_q) = \varinjlim_n \mathrm{GL}_n(\mathbb{F}_q)$, and $\mathrm{GL}_{\infty} (\mathbb{F}_q)$ has a subgroup $\mathrm{SL}_{\infty} (\mathbb{F}_q) = \varinjlim_n \mathrm{SL}_n(\mathbb{F}_q)$. By mimicking the construction in the proof of the previous proposition and noting the following facts: a linear injection $f: \underline{S} \to \underline{T}$ with $|S| < |T|$ can extend to an element $\tilde{f} \in \mathrm{SL}_n(\underline{T})$ via viewing $\underline{S}$ as a proper subspace of $\underline{T}$; and $\tilde{f}$ can be regarded as an element in $G$ in a natural way, one can show that $\mathrm{GL}_{\infty} (\mathbb{F}_q)$ and $\mathrm{SL}_{\infty} (\mathbb{F}_q)$ also have orbit categories equivalent to $\VI^{\op}$. Consequently, the categories of discrete representations of these three linear groups can be identified.

\subsubsection{Simple and indecomposable injective objects}

Let $G$ be one of the following groups: $\mathrm{GL} (\mathbb{F}_q^{\N})$, $\mathrm{GL}_{\infty} (\mathbb{F}_q)$, and $\mathrm{SL}_{\infty} (\mathbb{F}_q)$. Let $\C$ be the skeletal subcategory of $\VI$ with objects $\mathbb{F}_q^n$, $n \in \N$.

Let $F(n)$ be the $RG$-module spanned by (as an $R$-module) all $\mathbb{F}_q$-linear injections from $\mathbb{F}_q^n$ to $\mathbb{F}_q^{\N}$. It is easy to check that $F(n)$ is a discrete representation of $G$, and corresponds to the representable functor $R\C(\mathbb{F}_q^n, -)$. We say that a discrete $RG$-module is finitely generated if it is a quotient of $\oplus_{n \in \N} F(n)^{d_n}$ such that $\sum_{n \in \N} d_n$ is finite. The category of finitely generated discrete $RG$-module is abelian by the locally Noetherian property of $R\C$ (see for instance \cite{GL0}), denoted by $RG \dmodule$. We thus have by Proposition \ref{equivalence of Z}:
\[
RG \dmodule \simeq \mathrm{sh}(\BC^{\op}, R) \simeq \C \module / \C \module^{\tor} \simeq \C \fdmod.
\]

Irreducible objects in $\mathrm{sh}(\BC^{\op}, R)$ have been classified in \cite[Theorem 1.1]{Nag2} for the non-describing case. Here we give a slightly more explicit description of these objects for an algebraically closed field of characteristic 0, relying on the work of Gan and Watterlond in \cite{GW}. For this purpose, we introduce a few notations, most of which follow \cite{GW}.

Let $\mathcal{C}_n$ be the set of isomorphism classes of irreducible cuspidal representations of $\mathrm{GL}_n(\mathbb{F}_q)$, and let $\mathcal{C} = \sqcup_{n \in \N} \mathcal{C}_n$. Let $\mathcal{P}_n$ be the set of partitions of $[n]$, and let $\mathcal{P} = \sqcup_{n \in \N}\mathcal{P}_n$. For $\rho \in \mathcal{C}_n$, we set $d(\rho) = n$. For a partition $\lambda: \lambda_1 \geqslant \lambda_2 \geqslant \ldots$, we set $|\lambda| = \lambda_1 + \lambda_2 + \ldots$. For any function $\boldsymbol{\lambda}: \mathcal{C} \to \mathcal{P}$, we let
\[
\| \boldsymbol{\lambda} \| = \sum _{\rho \in \mathcal{C}} d(\rho) |\boldsymbol{\lambda}(\rho)|.
\]
By \cite[Section 9]{Ze}, elements in $\mathrm{Irr}(\mathrm{GL}_n(\mathbb{F}_q))$ are parameterized by functions $\boldsymbol{\lambda}: \mathcal{C} \to \mathcal{P}$ such that $\| \boldsymbol{\lambda} \| = n$.

Now let $\iota$ be the trivial representation of $\mathrm{GL}_1(\mathbb{F}_q)$, and for any $\boldsymbol{\lambda}: \mathcal{C} \to \mathcal{P}$, we let $\boldsymbol{\lambda}(\iota)$ be the partition $\lambda_1 \geqslant \lambda_2 \geqslant \ldots.$ For any integer $n \geqslant \| \boldsymbol{\lambda} \| + \lambda_1$, we define a new function
\[
\boldsymbol{\lambda}[n]: \mathcal{C} \to \mathcal{P}
\]
such that $\boldsymbol{\lambda}[n](\iota)$ is the partition $n -\lambda_1 \geqslant \lambda_1 \geqslant \lambda_2 \geqslant \ldots$ and $\boldsymbol{\lambda}[n](\rho) = \boldsymbol{\lambda}(\rho)$ for $\rho \neq \iota$. Note that $\boldsymbol{\lambda}[n]$ corresponds to an irreducible representation of $\mathrm{GL}_n(\mathbb{F}_q)$.

With these notations, we have:

\begin{proposition}
There is a bijection between the set of functions $\boldsymbol{\lambda}: \mathcal{C} \to \mathcal{P}$ such that $\| \boldsymbol{\lambda} \| < \infty$ and the set of isomorphism classes of simple objects in $\Sh(\BC^{\op}, R)$. Explicitly, given such a function $\boldsymbol{\lambda}$, the corresponding simple object $V(\boldsymbol{\lambda})$ in $\Sh(\BC^{\op}, R)$ has the following structure: for $n < \| \boldsymbol{\lambda} \| + \lambda_1$, one has $V(\boldsymbol{\lambda})_n = 0$; and for $n \geqslant \| \boldsymbol{\lambda} \| + \lambda_1$, $V(\boldsymbol{\lambda})_n$ is the irreducible representation of $\mathrm{GL}_n(\mathbb{F}_q)$ corresponding to $\boldsymbol{\lambda}[n]$.
\end{proposition}

\begin{remark}
As a consequence of the above result, the Grothendieck group
\[
K_0(RG \DMod) \cong \mathscr{H} = \bigoplus_{n \in \N} K_0(R\mathrm{GL}_n(\mathbb{F}_q) \Mod).
\]
By \cite{Ze}, $\mathscr{H}$ has a natural positive self-adjoint Hopf algebra structure, whose multiplication and comultiplication are induced by the induction and restriction on representations of finite general linear groups respectively. By this isomorphism, one can also impose a positive self-adjoint Hopf algebra structure on $K_0(RG \DMod)$.
\end{remark}

\subsubsection{Representation stability}

Now we describe a stability property of discrete representations of the three linear groups. All arguments to establish Proposition \ref{stability of Aut(N)} are still valid for $G$ with slight modifications. Thus we have:

\begin{proposition}
Let $V$ be a discrete representation of $G$ over an arbitrary commutative ring $R$ and suppose that the presentation degree $N_V$ of $V$ is finite. Then one has
\[
V = \bigcup_{\substack{T \subset \mathbb{F}_q^{\N} \\ |T| \leqslant N_V }} V^{U_T}.
\]
Furthermore, $N_V$ is the minimal number such that the above identity hold.
\end{proposition}

\subsection{Automorphism groups of the linearly ordered set $\mathbb{Q}$}

Let $G = \Aut(\mathbb{Q}, \leqslant)$ be the group of all order-preserving permutations on the poset $(\mathbb{Q}, \leqslant)$. It can be topologized by letting
\[
\mathcal{U} = \{\Stab_G (S) \mid S \subset \mathbb{Q}, \, |S| < \infty \}
\]
be a fundamental system of open neighborhoods of the identity element. Then $\mathbf{S}_{\mathcal{U}}G$ is equivalent to $\OI^{\op}$ by \cite[Example D3.4.11]{Jo}.

Let $\C$ be the skeletal subcategory of $\OI$ with objects $[n]$, $n \in \N$. Given $n \in \N$, and let $F(n)$ be the free $R$-module spanned by the set of all order-preserving maps from $[n]$ to $\mathbb{Q}$. It is easy to see that $F(n)$ is a discrete $RG$-module, and furthermore, $F(n)$ corresponds to the free $\C$-module $P(n) = R\C([n], -)$. We say that a discrete $RG$-module $V$ is \textit{finitely generated} if there is a surjective homomorphism
\[
\bigoplus_{n \in \N} (F(n))^{c_n} \to V
\]
with $\sum_{n \in \N} c_n < \infty$. By the locally Noetherian property of $R\C$, the full subcategory consisting of finitely generated discrete $RG$-modules is an abelian subcategory of $RG \DMod$, denoted by $RG \dmodule$.

Simple objects in $\Sh(\BC^{\op}, R)$ have been classified in \cite[Theorem 1.2]{GL1}, so let us briefly summarize the result.

For $n \in \N$ and $i \in [n+1]$, denote by $\alpha_{n, i}$ the morphism from $[n]$ to $[n+1]$ such that $i \in [n+1]$ is not contained in the image of $\alpha_{n, i}$. For instance, $\alpha_{n, n+1}$ is the canonical inclusion $[n] \to [n+1]$. Now for $i \in [n]$, let $L_{n, i} = \langle \alpha_{n,i} - \alpha_{n, i+1} \rangle$ be the two-sided ideal of the category algebra $R\C$ generated by $\alpha_{n, i} - \alpha_{n, i+1}$, and define $L_n = \cap_{i \in [n]} L_{n,i}$. For $n = 0$, by convention we let $L_0 = R\C([0], -)$.

\begin{proposition} \label{simple OI-sheaves}
Let $R$ be a field. Then every simple (resp., indecomposable injective) object in $\Sh(\BC^{\op}, R)$ is isomorphic to $L_n$ (resp., $F(n))$ for a certain $n \in \N$.
\end{proposition}

\begin{proof}
The conclusion for simple objects follows from \cite[Theorem 4.9]{GL1}. The conclusion for indecomposable injective objects follows from Corollary \ref{finitely generated sheaves} as well as the classification of indecomposable injective $\C$-modules in \cite[Theorem 14.2]{GS} (see also \cite[Proposition 4.4]{GL1}).
\end{proof}

A representation stability result for $G$ is:

\begin{proposition}
Let $V$ be a discrete representation of $G$ over an arbitrary commutative ring $R$ and suppose that the presentation degree $N_V$ of $V$ is finite. Then one has
\[
V = \bigcup_{\substack{T \subset \mathbb{Q} \\ |T| \leqslant N_V }} V^{U_T}.
\]
Furthermore, $N_V$ is the minimal number such that the above identity hold.
\end{proposition}

\subsection{The automorphism group $\Aut(B_{\infty})$}

Let $G = \Aut(B_{\infty})$ be the automorphism group of the free Boolean algebra $B_{\infty}$ on a countable infinity of generators. We may impose a topology on $G$ by letting the pointwise stabilizers of finite subalgebras form a base of neighborhoods of the identity. This group is isomorphic to a group of self-homeomorphisms of the Cantor space equipped with a special topology. Furthermore, $\mathbf{S}_{\mathcal{U}}G$ is equivalent to $\mathrm{FS}$, the category of finite sets and surjections. For details, see \cite[Example D3.4.12]{Jo}.

Representation theory of $\FS^{\op}$ have been considered in \cite{Pr, PY, SS3}. However, compared to $\FI$, there are still many questions unsolved. For instance, we know that it is a Type II combinatorial category, and satisfies the condition (LN) specified in Subsection \ref{type ii cats} by \cite[Theorem 8.1.2]{SS3}. However, it is not clear to the authors whether the other two conditions (LS) and (NV) hold. Therefore, an classification of simple discrete representation of $G$, even if $R$ is a field of characteristic 0, is still unavailable yet.

Fortunately, one can still apply \cite[Theorem 3.12]{GL0} to deduce the following stability result.

\begin{proposition}
Let $V$ be a discrete representation of $G$ over an arbitrary commutative ring $R$ and suppose that the presentation degree $N_V$ of $V$ is finite. Then one has
\[
V = \bigcup_{\substack{T \subset B_{\infty} \\ |T| \leqslant N_V }} V^{U_T},
\]
where $T$ is a finite subalgebra of $B_{\infty}$. Furthermore, $N_V$ is the minimal number such that the above identity hold.
\end{proposition}

\subsection{Other examples}

Let $p$ be a fixed prime number, and let $R$ be a field whose characteristic is distinct from $p$. Recall that $\mathcal{Z}(p^n)$ is the category of finitely generated $\mathbb{Z}/p^n \mathbb{Z}$-modules and surjective module homomorphisms, and $\mathcal{Z}(p^{\infty})$ is the category of finite abelian $p$-groups and conjugacy classes of surjective group homomorphisms; see \cite{Pol}. Furthermore, by Proposition \ref{equivalence of Z}, if $\C$ is one of these two categories, one has
\[
\mathrm{sh}(\BC, R) \simeq \C^{\op} \fdmod.
\]
Note that simple objects in $\C^{\op} \fdmod$ are parameterized by the set of irreducible representations (up to isomorphism) of $\Aut(H)$ when $H$ ranges over all finite abelian groups of exponent dividing $p^n$ (for $\mathcal{Z}(p^n))$ or finite abelian $p$-groups (for $\mathcal{Z}(p^{\infty})$). By the above equivalence, this set also parameterizes all simple objects in $\Sh(\BC, R)$.

Another example is the category $\FI^m$, the product category of $m$ copies of $\FI$, whose representations are considered in \cite{Gad, LY2, Zen}. Let $R$ be a field of characteristic 0, $\C$ be the skeletal subcategory of $\FI^m$ with objects $[n_1] \times \ldots \times [n_m]$, $n_i \in \N$, and let $G$ be the product group of $m$ copies of $\Aut(\N)$. By Proposition \ref{equivalence of Z}, one has $\mathrm{sh}(\BC^{\op}, R) \cong RG \dmodule$. Consequently, isomorphism classes of simple objects in $\mathrm{sh}(\BC^{\op}, R)$ or in $RG \dmodule$ are parameterized by the set $\mathcal{P}^m$, the Cartesian product of $m$ copies of $\mathcal{P}$. Furthermore, to obtain all simple objects in $RG \dmodule$, one only needs to exhaust the products of $m$ simple objects in $R\Aut(\N) \dmodule$. Indecomposable injective objects in $RG \dmodule$ can be constructed similarly.

\end{document}